\documentclass[12pt,a4paper]{amsart}
\usepackage{amsmath,amsthm,amssymb,amscd}
\usepackage[arrow,matrix]{xy}
\usepackage{mathtools}
\usepackage{enumerate}
\usepackage{ascmac}
\usepackage{graphicx}
\usepackage{mathrsfs}
\usepackage{typearea}
\usepackage{ytableau}
\usepackage{braket}
\usepackage{caption}
\usepackage{comment}

\calclayout

\theoremstyle{definition}
\newtheorem{thm}{Theorem}[section]
\newtheorem{main}{Theorem}

\newtheorem{lemma}[thm]{Lemma}
\newtheorem{corollary}[thm]{Corollary}
\newtheorem{remark}[thm]{Remark}
\newtheorem{dfn}[thm]{Definition}

\newtheorem{prop}[thm]{Proposition}

\newtheorem{exa}[thm]{Example}

\newcommand{\mf}[1]{{\mathfrak{#1}}}
\newcommand{\mb}[1]{{\mathbf{#1}}}
\newcommand{\bb}[1]{{\mathbb{#1}}}
\newcommand{\mca}[1]{{\mathcal{#1}}}
\newcommand{\mr}[1]{{\mathrm{#1}}}
\newcommand{\msc}[1]{{\mathscr{#1}}}
\newcommand{\msf}[1]{{\mathsf{#1}}}

\newcommand{\da}[1]{\big\downarrow\raise.5ex\rlap{$\scriptstyle#1$}}

\makeatletter
\@namedef{subjclassname@2020}{\textup{2020} Mathematics Subject Classification}
\makeatother

\begin{document}

\title{$(q,t)$-chromatic symmetric functions}
\author{Tatsuyuki Hikita}
\address{\textsc{Research Institute for Mathematical Sciences, Oiwake Kita-Shirakawa Sakyo Kyoto 606-8502 JAPAN}}
\email{\texttt{thikita@kurims.kyoto-u.ac.jp}}
\date{}

\begin{abstract}
	By using level one polynomial representations of affine Hecke algebras of type $A$, we obtain a $(q,t)$-analogue of the chromatic symmetric functions of unit interval graphs which generalizes Syu Kato's formula for the chromatic symmetric functions of unit interval graphs. We show that at $q=1$, the $(q,t)$-chromatic symmetric functions essentially reduce to the chromatic quasisymmetric functions defined by Shareshian-Wachs, which in particular gives an algebraic proof of Kato's formula. We also give an explicit formula of the $(q,t)$-chromatic symmetric functions at $q=\infty$, which leads to a probability theoretic interpretation of $e$-expansion coefficients of chromatic quasisymmetric functions used in our proof of the Stanley-Stembridge conjecture.
		
	Moreover, we observe that the $(q,t)$-chromatic symmetric functions are multiplicative with respect to certain deformed multiplication on the ring of symmetric functions. We give a simple description of such multiplication in terms of the affine Hecke algebras of type $A$. We also obtain a recipe to produce $(q,t)$-chromatic symmetric functions from chromatic quasisymmetric functions, which actually makes sense for any oriented graphs. 
\end{abstract}

\maketitle

\section{Introduction}

The chromatic symmetric function $\mb{X}_\Gamma$ of a graph $\Gamma$ introduced by Stanley \cite{Sta95} is an actively studied graph invariants (see e.g. \cite{CKL20,MMS24,OS14}). Its $t$-analogue $\mb{X}_{\Gamma}(t)$ is introduced by Shareshian-Wachs \cite{SW12,SW16} and called chromatic quasisymmetric function, which is shown to be symmetric when $\Gamma$ is a unit interval graph. The number of unit interval graphs with $n$ vertices is given by the $n$-th Catalan number, and hence there are many bijections between unit interval graphs and other combinatorial objects such as Dyck paths, Hessenberg functions, and $312$-avoiding permutations. Correspondingly, $\mb{X}_{\Gamma}(t)$ is known to be related for example to the LLT polynomials \cite{CM18}, the cohomology of the regular semisimple Hessenberg varieties of type $A$ \cite{BC18,GP16,SW12,SW16}, and the characters of the Hecke algebras of type $A$ at Kazhdan-Lusztig basis elements \cite{CHSS16,Hai93}. 

In this paper, we introduce additional parameter $q$ to the chromatic quasisymmetric functions. Our starting point of this work is Syu Kato's formula \cite{Kat24} for the chromatic symmetric functions of unit interval graphs. In \cite{Kat23}, Kato constructed a smooth projective algebraic variety $\mf{X}_{\Gamma}$ with an action of $GL_n(\bb{C}[[z]])$ and a $GL_n(\bb{C}[[z]])$-equivariant morphism 
\begin{align*}
	\msf{m}_{\Gamma}:\mf{X}_{\Gamma}\rightarrow \mr{Gr}
\end{align*}
to the affine Grassmannian $\mr{Gr}$ of $GL_n$. By the decomposition theorem \cite{BBD81}, the pushforward of the constant sheaf on $\mf{X}_{\Gamma}$ by $\msf{m}_{\Gamma}$ is a direct sum of shifted irreducible $GL_n(\bb{C}[[z]])$-equivariant perverse sheaves on $\mr{Gr}$. Via the geometric Satake correspondence \cite{MV07}, this produces a graded $GL_n$-representation and \cite{Kat24} (see also \cite{Kam24}) showed that its graded character is given by $\mb{X}_{\Gamma}(t)$ up to certain shift of grading.

In \cite{Kat24}, Kato also gives a formula for $\mb{X}_\Gamma$ by using the affine Weyl group $\widetilde{\mf{S}}_{n}=\langle s_0,s_1,\ldots,s_{n-1}\rangle$ of type $A_{n-1}$. In the group algebra $\bb{Q}[\widetilde{\mf{S}}_{n}]$ of $\widetilde{\mf{S}}_{n}$, consider an element
\begin{align*}
	\overline{\msf{S}}^{(n)}_{i,e}\coloneqq 1+s_{i}+s_{i+1}s_{i}+\cdots +s_{n+e-1}s_{n+e-2}\cdots s_{i}\in\bb{Q}[\widetilde{\mf{S}}_{n}]
\end{align*}
for each $i-n\leq e<i$, where the indices of $s$ are considered to be an element of $\bb{Z}/n\bb{Z}$ and we have $n+e-i+1$ terms in this formula. We have a level one action of $\widetilde{\mf{S}}_{n}$ on the ring of Laurent polynomials of $n$-variables
\begin{align*}
	\bb{Q}(q)[X^{\pm1}]_{(n)}\coloneqq\bb{Q}(q)[X_1^{\pm1},\ldots,X_n^{\pm1}]
\end{align*}
over $\bb{Q}(q)$ given by 
\begin{align*}
	s_i(F(X_1,\ldots,X_n))=\begin{cases}
F(\ldots,X_{i+1},X_i,\ldots)&\mbox{ if }i=1,\ldots,n-1,\\
q^{-1}X_1X_n^{-1}F(qX_n,X_2,\ldots,X_{n-1},q^{-1}X_{1})&\mbox{ if }i=0.
\end{cases}
\end{align*}
For each unit interval graph $\Gamma$ of $n$ vertices, we may associate a sequence of integers $\msf{e}$ in 
\begin{align*}
\bb{E}_n\coloneqq\left\{\msf{e}=(\msf{e}(1),\ldots,\msf{e}(n))\in\bb{Z}^n\middle|0\leq \msf{e}(i)<i,\msf{e}(i)\leq \msf{e}(i+1)\right\}.
\end{align*}
Using this sequence, Kato's formula can be stated as
\begin{align*}
	\mb{X}_{\Gamma}^{(n)}=\overline{\msf{S}}^{(n)}_{1,\msf{e}(1)}\overline{\msf{S}}^{(n)}_{2,\msf{e}(2)}\cdots\overline{\msf{S}}^{(n)}_{n,\msf{e}(n)}(X_1X_2\cdots X_n)|_{q=1}\in\bb{Z}[X_1,\ldots,X_n]^{\mf{S}_n}
\end{align*}
where $\mb{X}_{\Gamma}^{(n)}$ is the truncation of $\mb{X}_{\Gamma}$ to a symmetric polynomial of $n$-variable which does not lose any information. 

One feature of this formula is that it is a priori not clear if the RHS is symmetric or not. The proof of this formula given in \cite{Kat24} is geometric and uses an iterated projective space bundle structure on $\mf{X}_{\Gamma}$. The geometric reason behind this symmetry is the existence of $GL_n(\bb{C}[[z]])$-action on $\mf{X}_{\Gamma}$, but each intermediate space does not have this symmetry in general. One of the aims of this paper is to give an algebraic proof of this formula and in particular the symmetricity of the RHS. 

Second feature of Kato's formula is that it is an equality of symmetric \textit{polynomials} instead of symmetric \textit{functions}. We show that this formula actually stabilizes to give an equality of symmetric functions. More precisely, we show that 
\begin{align*}
	\mb{X}_{\Gamma}^{(m)}=\overline{\msf{S}}^{(m)}_{1,\msf{e}(1)}\overline{\msf{S}}^{(m)}_{2,\msf{e}(2)}\cdots\overline{\msf{S}}^{(m)}_{n,\msf{e}(n)}(X_1X_2\cdots X_n)|_{q=1}\in\bb{Z}[X_1,\ldots,X_m]^{\mf{S}_m}
\end{align*}
for any $m\geq 0$, where we understand that $\overline{\msf{S}}^{(m)}_{i,e}=0$ if $e<i-n$ and $X_{m+1}=q^{-1}X_1,X_{m+2}=q^{-1}X_2,\ldots$ if $m<n$. We note that the geometry corresponding to $\mf{X}_{\Gamma}$ for general $m$ is studied by Kamnitzer in \cite{Kam24}.

Third feature of Kato's formula is that the RHS has an additional parameter $q$ which is not related to the parameter $t$ in the chromatic quasisymmetric functions, but comes from the theory of affine Lie algebras. It is hence natural to ask if one can lift Kato's formula to the chromatic quasisymmetric functions by using the affine Hecke algebras of type $A$. This is one of the problems we solve affirmatively in this paper. We note that the appearance of \textit{affine} Hecke algebras here is different from the appearance of \textit{finite} Hecke algebras in the work of Clearman-Hyatt-Shelton-Skandera \cite{CHSS16}, although the meaning of the parameter $t$ in the defining relations of the Hecke algebras is the same. 

Now we explain our first main result in more detail. Let $\msc{H}_{m}$ be the affine Hecke algebra of $GL_m$ defined over $\bb{Q}_{q,t}\coloneqq\bb{Q}(q,t)$. This algebra also has a level one action on $\bb{Q}_{q,t}[X^{\pm1}]_{(m)}$ and the generator $T_{i}\in\msc{H}_{m}$ for $i\in\bb{Z}/m\bb{Z}$ acts by
\begin{align*}
	T_i(F)=ts_i(F)+(t-1)\frac{F-s_i(F)}{1-X_iX_{i+1}^{-1}},
\end{align*}
where the action of $s_i$ is defined as above and we understand that $X_{i+km}=q^{-k}X_{i}$ for any $i=1,\ldots,m$ and $k\in\bb{Z}$. For each $i-m\leq e<i$, we define 
\begin{align*}
	\msf{S}^{(m)}_{i,e}\coloneqq 1+T_{i}^{-1}+T_{i+1}^{-1}T_{i}^{-1}+\cdots +T_{m+e-1}^{-1}T_{m+e-2}^{-1}\cdots T_{i}^{-1}\in\msc{H}_m
\end{align*}
 and set $\msf{S}^{(m)}_{i,e}=0$ if $e<i-m$. 
 
 For $0<m'<m$, we consider the truncation map
 \begin{align*}
 	\pi_{m,m'}:\bb{Q}_{q,t}[X]_{(m)}\rightarrow\bb{Q}_{q,t}[X]_{(m')}
 \end{align*}
 given by $X_{m'+1}=X_{m'+2}=\cdots=X_{m}=0$, where we set
 \begin{align*}
 	\bb{Q}_{q,t}[X]_{(m)}\coloneqq\bb{Q}_{q,t}[X_1,\ldots,X_m].
 \end{align*}
We denote by $\bb{Q}_{q,t}[X]_{(\infty)}$ the projective limit of $\bb{Q}_{q,t}[X]_{(m)}$ with respect to $\pi_{m,m'}$ in the category of graded vector spaces, where the grading of $\bb{Q}_{q,t}[X]_{(m)}$ is given by the degree of polynomials. We note that $\bb{Q}_{q,t}[X]_{(\infty)}$ has a natural ring structure. The subring of symmetric functions with variables $X_1,X_2,\ldots,$ and coefficients in $\bb{Q}_{q,t}$ is denoted by 
 \begin{align*}
 	\Lambda_{q,t}\subset\bb{Q}_{q,t}[X]_{(\infty)}.
 \end{align*}
 
 \begin{main}\label{Main_A}
 	Let $\Gamma$ be a unit interval graph corresponding to $\msf{e}\in\bb{E}_n$.
 	\begin{enumerate}
 		\item We have
 		\begin{align*}
	\mb{X}_{\Gamma}^{(m)}(q,t)\coloneqq t^{n(m-1)}\msf{S}^{(m)}_{1,\msf{e}(1)}\msf{S}^{(m)}_{2,\msf{e}(2)}\cdots\msf{S}^{(m)}_{n,\msf{e}(n)}(X_1X_2\cdots X_n)\in\bb{Z}[q^{-1},t^{\pm1}][X_1,\ldots,X_m]^{\mf{S}_m},
\end{align*}
        where we understand that $X_{i+km}=q^{-k}X_{i}$ for any $1\leq i\leq m$ and $k\in\bb{Z}$.
		\item For any $0<m'<m$, we have
		\begin{align*}
			\pi_{m,m'}\left(\mb{X}_{\Gamma}^{(m)}(q,t)\right)=\mb{X}_{\Gamma}^{(m')}(q,t).
		\end{align*}
		In particular, we obtain a well-defined symmetric function
		\begin{align*}
			\mb{X}_{\Gamma}(q,t)\coloneqq\varprojlim_{m}\mb{X}_{\Gamma}^{(m)}(q,t)\in\Lambda_{q,t}.
		\end{align*}
		\item We have
		\begin{align*}
			\mb{X}_{\Gamma}(1,t)=\mb{N}(\mb{X}_{\Gamma}(t)),
		\end{align*}
		where $\mb{N}$ is an algebra automorphism of $\Lambda_{q,t}$ characterized by 
		\begin{align*}
			\mb{N}(e_r(X))=t^{\frac{r(r-1)}{2}}e_r(X)
		\end{align*}
		for the $r$-th elementary symmetric function $e_r(X)\in\Lambda_{q,t}$.
 	\end{enumerate}
 \end{main}
 
We call $\mb{X}_{\Gamma}(q,t)$ here $(q,t)$-\textit{chromatic symmetric function} for unit interval graph $\Gamma$. This is a $(q,t)$-analogue of the chromatic symmetric function $\mb{X}_{\Gamma}$ and a $q$-analogue of the slightly modified chromatic quasisymmetric function $\mb{N}(\mb{X}_{\Gamma}(t))$. We note that concerning the elementary symmetric function expansion ($e$-expansion for short) of the chromatic quasisymmetric function of a unit interval graph, this modification is harmless.

Now we study several properties of $\mb{X}_{\Gamma}(q,t)$. It is natural to ask if the known properties of $\mb{X}_{\Gamma}(t)$ still hold or not for $\mb{X}_{\Gamma}(q,t)$. For example, the famous Stanley-Stembridge conjecture \cite{Sta95,SS93} (proved in \cite{Hik24}) and its refinement due to Shareshian-Wachs \cite{SW12,SW16} (still open) imply that $\mb{X}_{\Gamma}(t)$ should be $e$-positive for any unit interval graph $\Gamma$. However, this positivity does not literally hold for $\mb{X}_{\Gamma}(q,t)$ as one can see easily by calculating some examples. We remedy this situation by also introducing a $(q,t)$-analogue of the elementary symmetric functions. Interestingly (or not), we show that the coefficients of the expansion of $\mb{X}_{\Gamma}(q,t)$ in terms of the $(q,t)$-elementary symmetric functions are essentially the same as the $e$-expansion coefficients of $\mb{X}_{\Gamma}(t)$. In particular, the coefficients are independent of the parameter $q$.

We next consider the multiplicativity of $\mb{X}_{\Gamma}(q,t)$. It is well-known that the chromatic quasisymmetric functions satisfy 
\begin{align*}
	\mb{X}_{\Gamma\cup\Gamma'}(t)=\mb{X}_{\Gamma}(t)\mb{X}_{\Gamma'}(t),
\end{align*}
where $\Gamma\cup\Gamma'$ is the ordered disjoint union of two unit interval graphs $\Gamma$ and $\Gamma'$. It turns out that this kind of multiplicativity does not hold for $\mb{X}_{\Gamma}(q,t)$ under the usual multiplication of symmetric functions. However, we show that there exists unique deformation $\star:\Lambda_{q,t}\times\Lambda_{q,t}\rightarrow\Lambda_{q,t}$ of the multiplication on $\Lambda_{q,t}$ such that we have
\begin{align*}
	\mb{X}_{\Gamma\cup\Gamma'}(q,t)=\mb{X}_{\Gamma}(q,t)\star\mb{X}_{\Gamma'}(q,t).
\end{align*}
for any unit interval graphs $\Gamma$ and $\Gamma'$. This multiplication gives a commutative associative algebra structure with unit $1\in\Lambda_{q,t}$ and reduces to the usual multiplication at $q=1$. We call this \textit{quantum multiplication}, in view of its similarity to the theory of quantum cohomology.

Let us explain how to define the quantum multiplication. Since there exists an element $\msf{S}_{\Gamma}^{(m)}\in\msc{H}_{m}$ such that 
\begin{align*}
	\mb{X}_{\Gamma}^{(m)}(q,t)=\msf{S}_{\Gamma}^{(m)}(1),
\end{align*}
it is natural to assume that the quantum multiplication by a symmetric polynomial on $\bb{Q}_{q,t}[X]_{(m)}$ commutes with the $\msc{H}_m$-action. In fact, this commutativity and 
\begin{align*}
	\msf{S}_{\Gamma\cup\Gamma'}^{(m)}=\msf{S}_{\Gamma}^{(m)}\msf{S}_{\Gamma'}^{(m)}
\end{align*}
imply that 
\begin{align*}
	\mb{X}_{\Gamma\cup\Gamma'}^{(m)}(q,t)=\msf{S}_{\Gamma}^{(m)}(\mb{X}_{\Gamma'}(q,t))=\msf{S}_{\Gamma}^{(m)}(1\star\mb{X}_{\Gamma'}^{(m)}(q,t))=\mb{X}_{\Gamma}^{(m)}(q,t)\star\mb{X}_{\Gamma'}^{(m)}(q,t).
\end{align*}
In order to obtain operators commuting with the $\msc{H}_{m}$-action, it is natural to consider the action of the center of $\msc{H}_{m}$. Recall that by the Bernstein presentation of the affine Hecke algebras (see for example \cite{Lus89}), we have a polynomial subalgebra 
\begin{align*}
	\bb{Q}_{q,t}[Y^{\pm1}]_{(m)}\subset\msc{H}_{m}
\end{align*}
and the center is equal to $\bb{Q}_{q,t}[Y^{\pm1}]_{(m)}^{\mf{S}_m}$ (see (\ref{Eqn_Y_i}) for our convention on $Y_i$). Therefore, we try to find an isomorphism 
\begin{align*}
	\msf{q}_{(m)}:\bb{Q}_{q,t}[Y^{\pm1}]_{(m)}\rightarrow\bb{Q}_{q,t}[X^{\pm1}]_{(m)}
\end{align*}
as vector spaces over $\bb{Q}_{q,t}$ preserving the polynomial parts such that
\begin{align*}
	F(X)\star G(X)=\msf{q}_{(m)}^{-1}(F)\bullet G(X)
\end{align*}
for any $F(X),G(X)\in\bb{Q}_{q,t}[X^{\pm1}]_{(m)}$, where we write the $\msc{H}_m$-action by $\bullet$ in order to emphasize it. Since $1\in\bb{Q}_{q,t}[X^{\pm1}]_{(m)}$ should be a unit, we should have
\begin{align*}
	\msf{q}_{(m)}(F(Y))=F(Y)\bullet 1.
\end{align*}
for any $F(Y)\in\bb{Q}_{q,t}[Y^{\pm1}]_{(m)}$. By reversing the argument, we define $\msf{q}_{(m)}$ by this formula and then define quantum multiplication $\star$ on $\bb{Q}_{q,t}[X^{\pm1}]_{(m)}$ as above. We note that it is crucial to use the level one action of $\msc{H}_m$ here in order for $\msf{q}_{(m)}$ to be an isomorphism.

We now summarize basic properties of the map $\msf{q}_{(m)}$ and the quantum multiplication as our second main result of this paper. 

\begin{main}\label{Main_B}\leavevmode
\begin{enumerate}
	\item The map $\msf{q}_{(m)}$ is an isomorphism of $\bb{Q}_{q,t}$-vector spaces and satisfies
	\begin{align*}
		\msf{q}_{(m')}(\pi_{m,m'}(F(Y)))=\pi_{m,m'}(\msf{q}_{(m)}(F(Y)))
	\end{align*}
	for any $0<m'<m$ and $F(Y)\in\bb{Q}_{q,t}[Y]_{(m)}$. In particular, we obtain a well-defined linear isomorphism
	\begin{align*}
		\msf{q}:\bb{Q}_{q,t}[Y]_{(\infty)}\xrightarrow{\sim}\bb{Q}_{q,t}[X]_{(\infty)}.
	\end{align*}
	\item We define a commutative associative multiplication $\star$ on $\bb{Q}_{q,t}[X]_{(\infty)}$ by
	\begin{align*}
		F(X)\star G(X)\coloneqq\msf{q}\left(\msf{q}^{-1}(F)\cdot\msf{q}^{-1}(G)\right),
	\end{align*}
	where $\cdot$ is the usual multiplication on $\bb{Q}_{q,t}[Y]_{(\infty)}$. This multiplication preserves the subspace $\Lambda_{q,t}\subset\bb{Q}_{q,t}[X]_{(\infty)}$ and reduces to the usual multiplication at $q=1$ if one of $F$ or $G$ is in $\Lambda_{q,t}$.
	\item For any unit interval graph $\Gamma$, we have
	\begin{align*}
	\mb{X}_{\Gamma}(q,t)=\msf{q}\left(\mb{Y}_{\Gamma}(t)\right),
	\end{align*}
	where $\mb{Y}_{\Gamma}(t)\in\bb{Q}_{q,t}[Y]_{(\infty)}$ is the chromatic quasisymmetric function of $\Gamma$ written in the $Y$-variables. 
	\item For any partition $\lambda=(\lambda_1,\ldots,\lambda_l)$, we have
	\begin{align*}
		\msf{q}(e_{\lambda}(Y))=t^{\sum_i\lambda_i(\lambda_i-1)/2}e^{(q,t)}_{\lambda}(X),
	\end{align*}
	where we define the $(q,t)$-\textit{elementary symmetric function} $e^{(q,t)}_{\lambda}(X)\in\Lambda_{q,t}$ by
	\begin{align*}
		e^{(q,t)}_{\lambda}(X)\coloneqq e_{\lambda_1}(X)\star\cdots\star e_{\lambda_l}(X).
	\end{align*}
\end{enumerate}
\end{main}

We note that the multiplicativity of $\mb{X}_{\Gamma}(q,t)$ is a consequence of the third assertion and the multiplicativity of $\mb{Y}_{\Gamma}(t)$. We also remark that the RHS of the equality in the third assertion makes sense for any oriented graph $\Gamma$ by the work of Ellzey \cite{Ell17}. Therefore, one might be able to consider this formula as a definition of $(q,t)$-chromatic (quasi)symmetric function for any oriented graphs.

Finally, we try to seek for an application of the $(q,t)$-chromatic symmetric functions. Let us write $e$-expansion of $\mb{Y}_{\Gamma}(t)$ as
\begin{align*}
	\mb{Y}_{\Gamma}(t)=\sum_{\lambda\vdash n}c_{\lambda}(\Gamma;t)e_{\lambda}(Y).
\end{align*}
The above theorem implies that we have
\begin{align*}
	\mb{X}_{\Gamma}(q,t)=\sum_{\lambda\vdash n}t^{\sum_i\lambda_i(\lambda_i-1)/2}c_{\lambda}(\Gamma;t)e^{(q,t)}_{\lambda}(X).
\end{align*}
In particular, we may obtain linear relations between the coefficients $c_{\lambda}(\Gamma;t)$ by specializing $q$ to some values other than $q=1$ in $\mb{X}_{\Gamma}(q,t)$ and $e^{(q,t)}_{\lambda}(X)$. We show that everything can be made explicit at $q=\infty$. The results can be summarized as follows.

\begin{main}\label{Main_C}\leavevmode
\begin{enumerate}
	\item For any unit interval graph $\Gamma$ corresponding to $\msf{e}\in\bb{E}_n$, we have
	\begin{align*}
		\lim_{q\rightarrow\infty}\mb{X}_{\Gamma}(q,t)=t^{n(n-1)/2-|\msf{e}|}[n]_t!e_n(X),
	\end{align*}
	where we write $|\msf{e}|=\msf{e}(1)+\cdots+\msf{e}(n)$.
	\item For any partition $\lambda=(\lambda_1,\ldots,\lambda_l)$, we have
	\begin{align*}
		\lim_{q\rightarrow\infty}e^{(q,t)}_{\lambda}(X)=\frac{[n]_t!}{\prod_{i=1}^{l}[\lambda_i]_t!}e_n(X)
	\end{align*}
	\item We obtain
	\begin{align*}
		\sum_{\lambda\vdash n}t^{|\msf{e}|-|\msf{e}_{\lambda}|}\frac{c_{\lambda}(\Gamma;t)}{\prod_{i=1}^{l}[\lambda_i]_t!}=1,
	\end{align*}
	where we write $|\msf{e}_{\lambda}|=\sum_{i<j}\lambda_i\lambda_j$.
\end{enumerate}
\end{main}

We note that the second assertion is a consequence of the following Pieri type formula
\begin{align*}
	e_1(X)\star e_r(X)=(1-q^{-1})[r+1]_te_{r+1}(X)+q^{-1}e_1(X)e_r(X)
\end{align*}
which might be of independent interest. It seems likely that similar Pieri type formula exists for more general quantum multiplication of $e_r(X)$ and Schur functions, but we do not pursue this direction here. 

The third assertion above suggests a probability theoretic approach to study the coefficients $c_{\lambda}(\Gamma;t)$. In fact, this formula was the starting point of our proof of the Stanley-Stembridge conjecture in \cite{Hik24}.

\subsection*{Acknowledgements}

The author thanks Syu Kato for discussions and comments. 

\section{Preliminaries}

\subsection{Notations}

For each $n\in\bb{Z}$, we define $t$-integers $[n]_t$ and $t$-factorial $[n]_t!$ by 
\begin{align*}
	[n]_t\coloneqq\frac{1-t^{n}}{1-t},\hspace{1em} [n]_t!\coloneqq\prod_{i=1}^n[i]_t.
\end{align*}

We set $\bb{Q}_{q,t}\coloneqq\bb{Q}(q,t)$. For $m\in\bb{Z}_{>0}$, we write
\begin{align*}
	\bb{Q}_{q,t}[X]_{(m)}\coloneqq\bb{Q}_{q,t}[X_1,\ldots,X_m]\subset\bb{Q}_{q,t}[X^{\pm1}]_{(m)}\coloneqq\bb{Q}_{q,t}[X_1^{\pm1},\ldots,X_m^{\pm1}].
\end{align*}
In $\bb{Q}_{q,t}[X]_{(m)}$, we understand that 
\begin{align*}
	X_{i+km}=q^{-k}X_{i}
\end{align*}
for any $i=1,\ldots,m$ and $k\in\bb{Z}$. For $0<m'<m$, we denote by 
\begin{align*}
	\pi_{m,m'}:\bb{Q}_{q,t}[X]_{(m)}\rightarrow\bb{Q}_{q,t}[X]_{(m')}
\end{align*}
the truncation map given by substituting $X_{i}=0$ for $m'<i\leq m$. We consider the grading on $\bb{Q}_{q,t}[X]_{(m)}$ given by the degree of the polynomials and denote by $\bb{Q}_{q,t}[X]_{(m),d}$ the subspace spanned by homogeneous polynomials of degree $d$. We define the space of ``polynomial functions'' by
\begin{align*}
	\bb{Q}_{q,t}[X]_{(\infty)}\coloneqq\bigoplus_{d\geq0}\varprojlim_{m}\bb{Q}_{q,t}[X]_{(m),d},
\end{align*}
where the projective limit is taken with respect to $\pi_{m,m'}$. The ring of symmetric functions (\cite{Mac95}) with coefficients in $\bb{Q}_{q,t}$ is defined by
\begin{align*}
	\Lambda_{q,t}\coloneqq\bigoplus_{d\geq0}\varprojlim_{m}\bb{Q}_{q,t}[X]_{(m),d}^{\mf{S}_m}\subset\bb{Q}_{q,t}[X]_{(\infty)}.
\end{align*}
Let $\Lambda_{q,t}^{(m)}\coloneqq\bb{Q}_{q,t}[X]_{(m)}^{\mf{S}_m}$ be the ring of symmetric polynomials of $m$-variables. We denote by $e_r(X_1,\ldots,X_m)\in \Lambda_{q,t}^{(m)}$ the $r$-th elementary symmetric polynomial and $e_r(X)\in\Lambda_{q,t}$ the $r$-th elementary symmetric function. For a partition $\lambda = ( \lambda_1,\ldots,\lambda_l )$, we set
\begin{align*}
	e_{\lambda}(X)\coloneqq \prod_{i=1}^l e_{\lambda_i}(X) \in \Lambda_{q,t}.
\end{align*}

When we specialize some parameter $q$ at $q=a$ on $\bb{Q}_{q,t}[X^{\pm1}]_{(m)}$ for example, we implicitly consider the natural subalgebra $\bb{Q}(t)[q][X^{\pm1}]_{(m)}$ (or $\bb{Q}(t)[q^{-1}][X^{\pm1}]_{(m)}$ if $a=\infty$) and take the quotient by the ideal generated by $(q-a)$ (or $q^{-1}$ if $a=\infty$). 

\subsection{Unit interval graphs}

We briefly recall the notion of unit interval graphs. Let $\bb{A}_n$ be the set of area sequences (see \cite{AP18}) of unit interval graphs with $n$ vertices, i.e., 
\begin{align*}
	\bb{A}_n\coloneqq\left\{\msf{a}:[n]\rightarrow\bb{Z}\middle| 0\leq \msf{a}(i)<i, \msf{a}(i+1)\leq \msf{a}(i)+1\right\},
\end{align*}
where we set $[n]\coloneqq\{1,\ldots,n\}$. For $\msf{a}\in\bb{A}_n$, the corresponding unit interval graph $\Gamma_{\msf{a}}$ has vertices $[n]$ and oriented edges
\begin{align*}
	i-\msf{a}(i)\rightarrow i,i-\msf{a}(i)+1\rightarrow i,\ldots,i-1\rightarrow i
\end{align*}
for $i\in[n]$. 

There are many other sets which can be used to parametrize unit interval graphs. For example, the set of Hessenberg functions $\bb{H}_n$ given by 
\begin{align*}
	\bb{H}_n\coloneqq\left\{\msf{h}:[n]\rightarrow[n]\middle| \msf{h} (i)\geq i,\msf{h}(i)\leq \msf{h}(i+1)\right\}
\end{align*}
corresponds bijectively\footnote{By combining $\msf{a}\mapsto\Gamma_{\msf{a}}$, $\msf{h} \in\bb{H}_n$ also corresponds to unit interval graphs. This correspondence is transpose to the correspondence used in \cite{AN21}, but the chromatic quasisymmetric functions are the same.} to $\bb{A}_n$ by $\msf{h}(i)=i+\msf{a}(n+1-i)$.


We will also use another set $\bb{E}_n$ of sequence of integers defined by 
\begin{align*}
\bb{E}_n\coloneqq\left\{\msf{e}:[n]\rightarrow\bb{Z}\middle|0\leq \msf{e}(i)<i,\msf{e}(i)\leq \msf{e}(i+1)\right\}.
\end{align*}
We note that we have a bijection $\bb{E}_n\cong\bb{A}_n$ given by $\msf{a}(i)=i-1-\msf{e}(i)$ and $\bb{E}_n\cong\bb{H}_n$ given by $\msf{h}(i)=n-\msf{e}(n+1-i)$. For $\msf{e}\in\bb{E}_n$, we denote by $\Gamma_{\msf{e}}$ the unit interval graph corresponding to $\msf{e}\in\bb{E}_n\cong\bb{A}_n$. 

For two unit interval graphs $\Gamma$ and $\Gamma'$ with $n$ and $n'$ vertices, we can consider their ordered disjoint union $\Gamma\cup\Gamma'$, where the vertices of $\Gamma$ are labeled by $1,\ldots,n$ and the vertices of $\Gamma'$ are labeled by $n+1,\ldots,n+n'$ in $\Gamma\cup\Gamma'$. In terms of $\bb{A}_n$ or $\bb{E}_n$, this operation can be described as follows.

\begin{dfn}\label{Def_concat}
For $\msf{a}\in\bb{A}_n$ and $\msf{a}'\in\bb{A}_{n'}$, we define their concatenation $\msf{a}\cup\msf{a}'\in\bb{A}_{n+n'}$ by
\begin{align*}
	(\msf{a}\cup\msf{a}')(i)=\begin{cases}
		\msf{a}(i)&\mbox{ if }1\leq i\leq n,\\
		\msf{a}'(i-n)&\mbox{ if }n+1\leq i\leq n+n'.
	\end{cases}
\end{align*}

For $\msf{e}\in\bb{E}_n$ and $\msf{e}'\in\bb{E}_{n'}$, we define their concatenation $\msf{e}\cup\msf{e}'\in\bb{E}_{n+n'}$ by
\begin{align*}
	(\msf{e}\cup\msf{e}')(i)=\begin{cases}
		\msf{e}(i)&\mbox{ if }1\leq i\leq n,\\
		n+\msf{e}'(i-n)&\mbox{ if }n+1\leq i\leq n+n'.
	\end{cases}
\end{align*}
\end{dfn}

\begin{exa}\label{Exa_complete_graph}
Let $\msf{e}_n\in\bb{E}_n$ be the element given by
\begin{align*}
	\msf{e}_n(i)=0
\end{align*}
 for any $i\in[n]$. The corresponding graph $\Gamma_{\msf{e}_n}$ is the complete graph of $n$-vertices.	For a composition $\mu=(\mu_1,\ldots,\mu_l)$ of $n$, we write
\begin{align*}
	\msf{e}_{\mu}\coloneqq\msf{e}_{\mu_1}\cup\cdots\cup\msf{e}_{\mu_l}\in\bb{E}_n.
\end{align*} 
\end{exa}

\subsection{Chromatic quasisymmetric functions}

We next recall the notion of chromatic quasisymmetric functions defined by Shareshian-Wachs \cite{SW12,SW16} for labeled graphs and Ellzey \cite{Ell17} for oriented graphs. For an oriented graph $\Gamma$, we denote by $\mr{Col}(\Gamma)$ the set of proper colorings $\kappa:\Gamma\rightarrow\bb{Z}_{>0}$, i.e., maps from the set of vertices of $\Gamma$ to $\bb{Z}_{>0}$ such that for each edge $v\rightarrow w$, we have $\kappa(v)\neq \kappa(w)$. For $\kappa\in\mr{Col}(\Gamma)$, we denote by $\mr{asc}(\kappa)$ the number of edges $v\rightarrow w$ in $\Gamma$ such that $\kappa(v)<\kappa(w)$.

\begin{dfn}
The chromatic quasisymmetric function $\mb{X}_{\Gamma}(t)$ of an oriented graph $\Gamma$ is defined by 
\begin{align*}
	\mb{X}_{\Gamma}(t)=\sum_{\kappa\in\mr{Col}(\Gamma)}t^{\mr{asc}(\kappa)}\prod_{v\in\Gamma}X_{\kappa(v)}.
\end{align*}
\end{dfn}

It is known that $\mb{X}_{\Gamma}(t)$ is a symmetric function if $\Gamma$ is an unit interval graph. We use the indeterminant $t$ here since this corresponds to the $t$ in $\mathscr{H}_m$ later.

\begin{exa}
	For $\msf{e}_n\in\bb{E}_n$ as in Example~\ref{Exa_complete_graph}, we have
\begin{align*}
	\mb{X}_{\Gamma_{\msf{e}_n}}(t)=[n]_t!e_n(X).	
\end{align*}
For two oriented graph $\Gamma$ and $\Gamma'$, we have
\begin{align*}
	\mb{X}_{\Gamma\cup\Gamma'}(t)=\mb{X}_{\Gamma}(t)\mb{X}_{\Gamma'}(t).
\end{align*}
In particular, we have
\begin{align*}
	\mb{X}_{\Gamma_{\msf{e}_\mu}}(t)=\prod_{i=1}^{l}[\mu_i]_t! \cdot e_{\lambda}(X)
\end{align*}
for any composition $\mu=(\mu_1,\ldots,\mu_l)$, where $\lambda$ is the partition obtained by rearranging $\mu$.
\end{exa}

\subsection{Modular law}

One of the fundamental properties of chromatic quasisymmetric function is the characterization by the modular law given by the work of Abreu-Nigro \cite{AN21}. In terms of Hessenberg functions, it is given as follows. 

\begin{dfn}[Modular law]\label{Def_ml}
Let $V$ be a $\bb{Q}(t)$-vector space. We say that a map $\chi:\bb{H}_n\rightarrow V$ satisfies the \textit{modular law} if we have
\begin{align*}
	(1+t)\chi(\msf{h})=t\chi(\msf{h}')+\chi(\msf{h}'')
\end{align*}
for any triple $(\msf{h},\msf{h}',\msf{h}'')$ of Hessenberg functions satisfying one of the following conditions:
\begin{enumerate}
	\item There exists $i\in[n-1]$ such that $\msf{h}(i-1)<\msf{h}(i)<\msf{h}(i+1)$ and $\msf{h}(\msf{h}(i))=\msf{h}(\msf{h}(i)+1)$. Moreover, we have $\msf{h}'(j)=\msf{h}''(j)=\msf{h}(j)$ for $j\neq i$ and $\msf{h}'(i)=\msf{h}(i)-1$ and $\msf{h}''(i)=\msf{h}(i)+1$.
	\item There exists $i\in [n-1]$ such that $\msf{h}(i+1)=\msf{h}(i)+1$ and $\msf{h}^{-1}(i)=\emptyset$. Moreover, we have $\msf{h}'(j)=\msf{h}''(j)=\msf{h}(j)$ for $j\neq i,i+1$ and $\msf{h}'(i)=\msf{h}'(i+1)=\msf{h}(i)$ and $\msf{h}''(i)=\msf{h}''(i+1)=\msf{h}(i+1)$.
\end{enumerate}
Here we understand that $\msf{h}(0)=1$ if $i=1$ in (i).
\end{dfn}

In terms of $\bb{E}_n$, the modular law can be written as follows.

\begin{lemma}\label{Lem_emodular}
Let $\mca{R}$ be a $\bb{Q}(t)$-vector space and let $\chi:\bb{E}_n\rightarrow V$ be a map. The composition $\bb{H}_n\cong\bb{E}_n\xrightarrow{\chi} V$ satisfies the modular law if and only if we have
\begin{align*}
	(1+t)\chi(\msf{e})=t\chi(\msf{e}')+\chi(\msf{e}'')
\end{align*}
for any triple $(\msf{e},\msf{e}',\msf{e}'')$ in $\bb{E}_n$ satisfying one of the following conditions:
\begin{enumerate}
	\item There exists $1<i\leq n$ such that $\msf{e}(i-1)<\msf{e}(i)<\msf{e}(i+1)$ and $\msf{e}(\msf{e}(i))=\msf{e}(\msf{e}(i)+1)$. Moreover, we have $\msf{e}'(j)=\msf{e}''(j)=\msf{e}(j)$ for $j\neq i$ and $\msf{e}'(i)=\msf{e}(i)+1$ and $\msf{e}''(i)=\msf{e}(i)-1$.
	\item There exists $i\in[n-1]$ such that $\msf{e}(i+1)=\msf{e}(i)+1$ and $\msf{e}^{-1}(i)=\emptyset$. Moreover, we have $\msf{e}'(j)=\msf{e}''(j)=\msf{e}(j)$ for $j\neq i,i+1$ and $\msf{e}'(i)=\msf{e}'(i+1)=\msf{e}(i+1)$ and $\msf{e}''(i)=\msf{e}''(i+1)=\msf{e}(i)$.
\end{enumerate}
Here we understand that $\msf{e}(n+1)=n-1$ if $i=n$ in (i).
\end{lemma}

\begin{proof}
	This follows from the definition by replacing $i$ with $n+1-i$ in the first condition and replacing $i$ with $n-i$ in the second condition in Definition~\ref{Def_ml}.
\end{proof}

\begin{thm}[Abreu-Nigro \cite{AN21}]\label{Thm_AN}
	The map
\begin{align*}
	\chi:\bigsqcup_{n\geq0}\bb{E}_n\rightarrow\Lambda_{q,t}
\end{align*}
given by $\chi(\msf{e})=\mb{X}_{\Gamma_{\msf{e}}}(t)$ is the unique map that satisfies the following conditions:
	\begin{enumerate}
		\item $\chi$ satisfies the modular law.
		\item $\chi(\msf{e}\cup \msf{e}')=\chi(\msf{e})\chi(\msf{e}')$ for any $\msf{e}\in\bb{E}_n$ and $\msf{e}'\in\bb{E}_{n'}$.
		\item $\chi(\msf{e}_n)=[n]_t!\cdot e_n(X)$.
	\end{enumerate}
\end{thm}

\begin{remark}\label{Rem_AN}
Abreu-Nigro \cite{AN21} actually show that if $\chi$ satisfies the modular law, then it is uniquely determined by the values at $\msf{e}_{\mu}$ for all compositions $\mu$. 
\end{remark}

\subsection{Affine Hecke algebras}

\begin{dfn}\label{def:AHA}
Let $\mathscr{H}_m$ be the affine Hecke algebra of $GL_{m}$ over $\bb{Q}_{q,t}$. Namely, $\mathscr{H}_m$ is the $\bb{Q}_{q,t}$-algebra generated by $T_i$ for $i\in\bb{Z}/m\bb{Z}$ and $\Pi^{\pm1}$ satisfying the following relations:
\begin{itemize}
	\item $(T_i-t)(T_i+1)=0$,
	\item $T_iT_{i+1}T_i=T_{i+1}T_iT_{i+1}$ if $m>2$,
	\item $T_iT_j=T_jT_i$ if $j\neq i\pm1$, 
	\item $\Pi T_i=T_{i+1}\Pi$.
\end{itemize}
\end{dfn}


We consider the level one polynomial representation of $\mathscr{H}_m$ on
$\bb{Q}_{q,t}[X^{\pm1}]_{(m)}$ defined by 
\begin{align}
	T_i\bullet F&=ts_i(F)+(t-1)\frac{F-s_i(F)}{1-X_iX_{i+1}^{-1}},\label{Eqn_commTX}\\
	\Pi\bullet F&=X_1F(X_2,\ldots,X_{m},X_{m+1}),\label{Eqn_commPiX}
\end{align}
for any $F=F(X_1,\ldots,X_m)\in\bb{Q}_{q,t}[X^{\pm1}]_{(m)}$. Here, we set
\begin{align*}
	s_i(F)=\begin{cases}
F(\ldots,X_{i+1},X_i,\ldots)&\mbox{ if }i=1,\ldots,m-1,\\
X_1X_0^{-1}\cdot F(X_0,X_2,\ldots,X_{m-1},X_{m+1})&\mbox{ if }i=0.
\end{cases}
\end{align*}
and recall that $X_0=qX_m$ and $X_{m+1}=q^{-1}X_1$.

\begin{remark}
We remark that the action of $s_0$ and $\Pi$ is twisted, and this twist corresponds to the level one action from a viewpoint of affine Lie algebras \cite{Kac90} instead of the usual (level zero) action employed in \cite{Che05}.
\end{remark}

\begin{lemma}\label{Lem_AHA_action}
For any $k,l\in\bb{Z}$ and $i=1,\ldots,m-1$, we have
\begin{align*}
	T_i\bullet X_i^kX_{i+1}^l&=\begin{cases}
		tX_i^lX_{i+1}^k+(t-1)(X_i^{l-1}X_{i+1}^{k+1}+X_i^{l-2}X_{i+1}^{k+2}+\cdots+X_i^kX_{i+1}^l)&\mbox{ if }l\geq k,\\
		X_i^lX_{i+1}^k-(t-1)(X_i^{l+1}X_{i+1}^{k-1}+X_i^{l+2}X_{i+1}^{k-2}+\cdots+X_i^{k-1}X_{i+1}^{l+1})&\mbox{ if }l<k.
	\end{cases}\\
	T_i^{-1}\bullet X_i^kX_{i+1}^l&=\begin{cases}
		X_i^lX_{i+1}^k+(1-t^{-1})(X_i^{l-1}X_{i+1}^{k+1}+X_i^{l-2}X_{i+1}^{k+2}+\cdots+X_i^{k+1}X_{i+1}^{l-1})&\mbox{ if }l> k,\\
		t^{-1}X_i^lX_{i+1}^k-(1-t^{-1})(X_i^{l+1}X_{i+1}^{k-1}+X_i^{l+2}X_{i+1}^{k-2}+\cdots+X_i^{k}X_{i+1}^{l})&\mbox{ if }l\leq k.
	\end{cases}
\end{align*}
In particular, $T_i^{\pm1}$ with $i\neq0$ preserves the subspace $\bb{Q}_{q,t}[X]_{(m)}\subset\bb{Q}_{q,t}[X^{\pm1}]_{(m)}$
\end{lemma}

\begin{proof}
This follows from straightforward calculations using (\ref{Eqn_commTX}) and
\begin{align*}
	T_i^{-1}\bullet F=s_i(F)+(1-t^{-1})\frac{X_iX_{i+1}^{-1}F-s_i(F)}{1-X_iX_{i+1}^{-1}}
\end{align*}
for any $F\in\bb{Q}_{q,t}[X^{\pm1}]_{(m)}$. The latter assertion follows by the first assertion and 
\begin{align*}
	T_i\bullet X_jF=X_j(T_i\bullet F)
\end{align*}
for any $j\neq i,i+1$.
\end{proof}

For $i=1,\ldots,m$, we set
\begin{align}\label{Eqn_Y_i}
	Y_i\coloneqq t^{m-i}T_{i-1}T_{i-2}\cdots T_1\Pi T_{m-1}^{-1}T_{m-2}^{-1}\cdots T_i^{-1} \in \mathscr{H}_m.
\end{align}
Note that our $Y_i$ essentially corresponds to $Y_i^{-1}$ in the usual convention for DAHA (see e.g. \cite{Che05}). It is well-known that the elements $Y_1,\ldots,Y_m$ mutually commute and satisfy
\begin{align}\label{Eqn_commTY}
	T_iF(Y)=s_i(F(Y))T_i+(t-1)\frac{F(Y)-s_i(F(Y))}{1-Y_iY_{i+1}^{-1}}
\end{align}
for any Laurent polynomial $F(Y)$ in $Y_1,\ldots,Y_m$ and $i=1,\ldots,m-1$. In particular, we obtain $\bb{Q}_{q,t}[Y^{\pm1}]_{(m)}\subset\msc{H}_m$. 

\section{Quantum multiplications}

In this section, we define and study the quantum multiplication on $\bb{Q}_{q,t}[X]_{(\infty)}$.

\subsection{The map $\msf{q}_{(m)}$}

We define a $\bb{Q}_{q,t}$-linear map $\msf{q}_{(m)}:\bb{Q}_{q,t}[Y^{\pm1}]_{(m)} \rightarrow\bb{Q}_{q,t}[X^{\pm1}]_{(m)}$ by
\begin{align*}
	\msf{q}_{(m)}(F(Y))=F(Y)\bullet1,
\end{align*}
where $\bullet$ means the $\mathscr{H}_m$-action on $\bb{Q}_{q,t}[X^{\pm1}]_{(m)}$ and $1\in\bb{Q}_{q,t}[X^{\pm1}]_{(m)}$. 

\begin{lemma}\label{Lem_iota}
The map	$\msf{q}_{(m)}$ is an isomorphism as $\bb{Q}_{q,t}$-vector spaces. This also restricts to an isomorphism $\msf{q}_{(m)}:\bb{Q}_{q,t}[Y]_{(m)} \xrightarrow{\sim}\bb{Q}_{q,t}[X]_{(m)}$ of $\bb{Q}_{q,t}$-vector spaces.
\end{lemma}

\begin{proof}
Since $T_1^{\pm1},\ldots,T_{m-1}^{\pm1}$ preserve the degree and $\Pi$ increases the degree by one (and hence each $Y_1,\ldots,Y_{m}$ increases the degree by one), the map $\msf{q}_{(m)}$ preserves the degree. Hence it is enough to check that $\msf{q}_{(m)}$ is an isomorphism on the degree $d$ parts. 

Recall that $T_1^{\pm1},\ldots,T_{m-1}^{\pm1}$, and $\Pi$ preserve the polynomial part $\bb{Q}_{q,t}[X]_{(m)}$ by Lemma~\ref{Lem_AHA_action} and (\ref{Eqn_commPiX}). This implies that $\msf{q}_{(m)}$ preserves the polynomial parts. Since
\begin{align*}
	t^{-n(n-1)/2}Y_1\cdots Y_m=\Pi^m
\end{align*}
acts by the multiplication by $q^dX_1\cdots X_m$ on the degree $d$ part and the translation of the polynomial part $\bb{Q}_{q,t}[X]_{(m)}$ by the action of $\{\Pi^{-k}\}_{k\geq 0}$ exhausts $\bb{Q}_{q,t}[X^{\pm1}]_{(m)}$. Therefore, it is enough to prove that for any $d\in\bb{Z}_{\geq0}$, the map $\msf{q}_{(m)}$ induces an isomorphism 
\begin{align*}
	\msf{q}_{(m)}:\bb{Q}_{q,t}[Y]_{(m),d} \xrightarrow{\sim}\bb{Q}_{q,t}[X]_{(m),d}
\end{align*}
between finite dimensional $\bb{Q}_{q,t}$-vector spaces of the same dimension. Under the specialization $q=t=1$, the action of $Y_i$ is given by the multiplication by $X_i$. Since the specialization does not increase the rank of a matrix, we conclude that $\msf{q}_{(m)}$ is also surjective, and hence an isomorphism.
\end{proof}

\begin{lemma}\label{Lem_sym_to_sym}
A Laurent polynomial $F(Y)\in\bb{Q}_{q,t}[Y^{\pm1}]_{(m)}$ is symmetric if and only if $\msf{q}_{(m)}(F(Y))\in\bb{Q}[X^{\pm1}]_{(m)}$ is symmetric.
\end{lemma}

\begin{proof}
For $F(Y)\in\bb{Q}_{q,t}[Y^{\pm1}]_{(m)}$, we define $T_i\bullet F(Y)\in\bb{Q}_{q,t}[Y^{\pm1}]_{(m)}$ by the same formula as (\ref{Eqn_commTX}) by replacing $X$ by $Y$. By the comparison of (\ref{Eqn_commTY}) and (\ref{Eqn_commTX}), we find that 
\begin{align}
	\msf{q}_{(m)}(T_i\bullet F(Y))=T_iF(Y)\bullet1=T_i\bullet\msf{q}_{(m)}(F(Y))\label{eqn_TXY}
\end{align}
for each $i=1,\ldots,m-1$. Note that $F \in \bb{Q}_{q,t}[X^{\pm1}]_{(m)}$ or $F \in \bb{Q}_{q,t}[Y^{\pm1}]_{(m)}$ is symmetric if and only if 
\begin{align*}
	T_i\bullet F=tF
\end{align*}
for any $i=1,\ldots,m-1$. Thus, $F(Y)$ is symmetric if and only if $T_i\bullet F(Y) = tF(Y)$ for $i =1,\ldots,m-1$, that is equivalent to $T_i\bullet\msf{q}_{(m)}(F(Y)) = t\msf{q}_{(m)}(F(Y))$ for $i =1,\ldots,m-1$ by (\ref{eqn_TXY}). Thus, we conclude the result.
\end{proof}

\begin{lemma}\label{Lem_iota_elem}
We have
\begin{align*}
	\msf{q}_{(m)}(e_r(Y_1,\ldots,Y_m))=t^{r(r-1)/2}e_r(X_1,\ldots,X_m).
\end{align*}
\end{lemma}

\begin{proof}
It is enough to prove that for any subset $I\subset\{1,\ldots,m\}$, we have 
\begin{align*}
	\msf{q}_{(m)}\left(\prod_{i\in I}Y_i\right)=t^{|I|(|I|-1)/2}\prod_{i\in I}X_i.
\end{align*}
We prove this by induction on $|I|=l$. The initial case $l = 0$ is trivial. Let $I=\{i_1,\ldots,i_{l+1}\}$ with $i_1<i_2<\cdots<i_{l+1}$. By induction hypothesis, we have $\msf{q}_{(m)}(Y_{i_1}\cdots Y_{i_l})=t^{l(l-1)/2}X_{i_1}\cdots X_{i_l}$. Since $T_i(X_i)=X_{i+1}$ and $T_i(F)=tF$ if $s_iF=F$, we calculate
\begin{align*}
	Y_{i_{l+1}}\cdot Y_{i_1}\cdots Y_{i_l}\bullet1&=t^{l(l-1)/2}Y_{i_{l+1}}(X_{i_1}\cdots X_{i_l})\\
	&=t^{l(l-1)/2+n-i_{l+1}}T_{i_{l+1}-1}\cdots T_1\Pi T_{n-1}^{-1}\cdots T_{i_{l+1}}^{-1}(X_{i_1}\cdots X_{i_l})\\
	&=t^{l(l-1)/2}T_{i_{l+1}-1}\cdots T_1\Pi(X_{i_1}\cdots X_{i_l})\\
	&=t^{l(l-1)/2}T_{i_{l+1}-1}\cdots T_1(X_1X_{i_1+1}\cdots X_{i_l+1})\\
	&=t^{l(l-1)/2}T_{i_{l+1}-1}\cdots T_{i_1}(X_{i_1}X_{i_{1}+1}\cdots X_{i_{l}+1})\\
	&=t^{l(l-1)/2+1}T_{i_{l+1}-1}\cdots T_{i_1+1}(X_{i_1}X_{i_{1}+1}\cdots X_{i_{l}+1})\\
	&=\cdots\\
	&=t^{l(l+1)/2}X_{i_1}\cdots X_{i_{l+1}}.
\end{align*} 
This completes the proof.
\end{proof}

\subsection{Quantum multiplications}

In this section, we introduce a new multiplication $\star$ on $\bb{Q}[X^{\pm1}]_{(m)}$ by twisting the usual multiplication $\cdot$ on $\bb{Q}_{q,t}[Y^{\pm1}]_{(m)}$ by $\msf{q}_{(m)}$.

\begin{dfn}\label{Def_qm}
	We define a multiplication $\star$ on $\bb{Q}_{q,t}[X^{\pm1}]_{(m)}$ by
	\begin{align*}
		F\star G\coloneqq\msf{q}_{(m)}\left(\msf{q}_{(m)}^{-1}(F)\cdot\msf{q}_{(m)}^{-1}(G)\right)
	\end{align*}
for any $F,G\in\bb{Q}_{q,t}[X^{\pm1}]_{(m)}$.
\end{dfn}

We note that one can write this multiplication as
\begin{align*}
	F\star G=(\msf{q}_{(m)}^{-1}(F)\cdot\msf{q}_{(m)}^{-1}(G))\bullet1=\msf{q}_{(m)}^{-1}(F)\bullet G.
\end{align*}
It is clear from the definition that $\star$ and the usual addition induce another commutative associative $\bb{Q}_{q,t}$-algebra structure on $\bb{Q}_{q,t}[X^{\pm1}]_{(m)}$ with the same identity $1$ as the usual algebra structure.

\begin{lemma}\label{Lem_star_commute}
If $F\in\bb{Q}_{q,t}[X^{\pm1}]_{(m)}$ is symmetric, then the linear endomorphism 
\begin{align*}
	F\star(-):\bb{Q}_{q,t}[X^{\pm1}]_{(m)}\rightarrow\bb{Q}_{q,t}[X^{\pm}]_{(m)}
\end{align*}
commutes with the $\mathscr{H}_m$-action.
\end{lemma}

\begin{proof}
By Lemma~\ref{Lem_sym_to_sym}, it is enough to show that if $F\in\bb{Q}_{q,t}[Y^{\pm1}]_{(m)}\subset\mathscr{H}_m$ is symmetric, then $F$ lies in the center of $\mathscr{H}_m$. This is well-known (see e.g. \cite{Lus89}) and also follows from (\ref{Eqn_commTY}) using the fact that $\mathscr{H}_m$ is generated by $T_1,\ldots,T_{m-1}$, and $Y_1^{\pm1}$.
\end{proof}

\begin{prop}\label{Prop_qmult_q=1}
When we specialize at $q=1$, the product $F\star G$ reduces to the usual product $F\cdot G$ for symmetric $F\in\bb{Q}[X^{\pm1}]_{(m)}$ and every $G\in\bb{Q}[X^{\pm1}]_{(m)}$.
\end{prop}

\begin{proof}
We note that at $q=1$, we have $X_{i+km}=X_{i}$ and hence the usual product by $F$ also commutes with the $\mathscr{H}_m$-action since we have 
\begin{align*}
	s_i(F\cdot G)&=F\cdot s_i(G),\\
	\Pi\bullet(F\cdot G)&=F\cdot(\Pi\bullet G)
\end{align*}
for every $G\in\bb{Q}_{q,t}[X^{\pm1}]_{(m)}$ and $i=0,1,\ldots,m-1$. Therefore, we obtain
\begin{align*}
	F\star G=G\star F=\msf{q}_{(m)}^{-1}(G)\bullet (F\cdot1)=F\cdot(\msf{q}_{(m)}^{-1}(G)\bullet1)=F\cdot G
\end{align*}
as required.
\end{proof}

\subsection{Stabilization}

By Lemma~\ref{Lem_iota}, $\bb{Q}_{q,t} [X]_{(m)}$ is stable under the multiplication $\star$. Applying Lemma~\ref{Lem_sym_to_sym}, we conclude that each of
\begin{align*}
	\Lambda^{(m)}_{q,t}\subset\bb{Q}_{q,t} [X]_{(m)}\subset\bb{Q}_{q,t} [X^{\pm1}]_{(m)}
\end{align*}
is stable under the quantum multiplication $\star$. We next show that $\star$ induces a new ring structure on $\bb{Q}_{q,t} [X]_{(\infty)}$. 

\begin{lemma}\label{Lem_divisibility}
For each $F\in\bb{Q}_{q,t} [X]_{(m)}$ that is divisible by $X_i$, we have
\begin{enumerate}
\item $T_i\bullet F$ is divisible by $X_{i+1}$ for $1 \leq i < m$,
\item $T_{i-1}^{-1}\bullet F $ is divisible by $X_{i-1}$ for $1 < i \leq m$.
\end{enumerate}
\end{lemma}
\begin{proof}
Fix $1 \leq i < m$.
By Lemma~\ref{Lem_AHA_action}, we have
\begin{align*}
	T_i \bullet ( X_i^k X_{i+1}^{l} ) = tX_i^lX_{i+1}^k+(t-1)(X_i^{l-1}X_{i+1}^{k+1}+X_i^{l-2}X_{i+1}^{k+2}+\cdots+X_i^kX_{i+1}^l)
\end{align*}
for $l > k \ge 0$ and 
\begin{align*}
	T_i \bullet ( X_i^k X_{i+1}^{l} ) = X_i^lX_{i+1}^k-(t-1)(X_i^{l+1}X_{i+1}^{k-1}+X_i^{l+2}X_{i+1}^{k-2}+\cdots+X_i^{k-1}X_{i+1}^{l+1})
\end{align*}
for $k \ge l \ge 0$. In both of the cases, $k > 0$ implies that $T_i \bullet ( X_i^k X_{i+1}^{l} )$ is divisible by $X_{i+1}$. This is the first assertion.

By Lemma~\ref{Lem_AHA_action}, we have
\begin{align*}
	T_i^{-1} \bullet ( X_i^k X_{i+1}^{l} ) = X_i^lX_{i+1}^k+(1-t^{-1})(X_i^{l-1}X_{i+1}^{k+1}+X_i^{l-2}X_{i+1}^{k+2}+\cdots+X_i^{k+1}X_{i+1}^{l-1})
\end{align*}
for $l > k \geq 0$ and
\begin{align*}
	T_i^{-1} \bullet ( X_i^k X_{i+1}^{l} ) = t^{-1}X_i^lX_{i+1}^k-(1-t^{-1})(X_i^{l+1}X_{i+1}^{k-1}+X_i^{l+2}X_{i+1}^{k-2}+\cdots+X_i^{k}X_{i+1}^{l})
\end{align*}
for $k \ge l \ge 0$. In both of the cases, $l > 0$ implies that $T_i^{-1} \bullet ( X_i^k X_{i+1}^{l} )$ is divisible by $X_{i}$. This is the second assertion.
\end{proof}

\begin{prop}\label{Prop_stability}
For any $F(Y)\in\bb{Q}_{q,t} [Y]_{(m)}$, $G(X)\in\bb{Q}_{q,t} [X]_{(m)}$, and $0<m'<m$, we have
\begin{align*}
	\pi_{m,m'}(F(Y)\bullet G(X))=\pi_{m,m'}(F(Y))\bullet\pi_{m,m'}(G(X)).
\end{align*}
\end{prop}

\begin{proof}
Since we have $\pi_{m.m'}=\pi_{m'+1,m'}\circ\cdots\circ\pi_{m,m-1}$, it is enough to consider the case $m'=m-1$. It is also enough to consider the case $F(Y)=Y_i$ for some $i=1,\ldots,m$. 

We first note that $Y_m\bullet G(X)$ is divisible by $X_m$ for any $G(X)\in\bb{Q}_{q,t} [X]_{(m)}$. Since $Y_m = T_{m-1} \cdots T_1 \Pi$ and $\Pi\bullet G(X)$ is divisible by $X_1$, we deduce that $Y_m \bullet G(X)$ is divisible by $X_m$ by repeated applications of Lemma~\ref{Lem_divisibility}. In particular, this proves the case $i=m$.

We assume $i<m$. We show that if $G(X)$ is divisible by $X_m$, then $Y_i\bullet G(X)$ is also divisible by $X_m$. Recall that $T_i$ commutes with the multiplication by $X_j$ for $j\neq i,i+1$. Hence $T_{m-2}^{-1}\cdots T_i^{-1}\bullet G(X)$ is divisible by $X_m$. By Lemma~\ref{Lem_divisibility}, we find that $T_{m-1}^{-1}\cdots T_i^{-1}\bullet G(X)$ is divisible by $X_{m-1}$. It follows that $\Pi T_{m-1}^{-1}\cdots T_i^{-1}\bullet G(X)$ is divisible by $X_m$. Since the action of $T_{i-1},\cdots,T_1$ commutes with the multiplication by $X_m$, we deduce that $Y_i\bullet G(X)$ is also divisible by $X_m$.

Therefore, we may assume that $G(X)$ does not contain $X_m$ and hence $\pi_{m,m-1}(G(X))=G(X)$ considered as an element of $\bb{Q}_{q,t} [X]_{(m-1)}$. Since each of $T_i^{-1}$ $(1 \leq i < m-1)$ does not introduce $X_m$ by applying it to a polynomial without $X_m$, we deduce $T_{m-2}^{-1}\cdots T_{i}^{-1}\bullet G(X)$ does not contain $X_m$. We expand it by monomials as
\begin{align*}
T_{m-2}^{-1}\cdots T_{i}^{-1}\bullet G(X)=\sum_{\mb{a}=(a_1,\ldots,a_{m-1})}c_{\mb{a}}X_1^{a_1}\cdots X_{m-1}^{a_{m-1}} 
\end{align*}
for some $c_{\mb{a}}\in\bb{Q}_{q,t}$. Since the action of $T_1,\ldots,T_{m-2}$ are the same for $\bb{Q}_{q,t} [X]_{(m)}$ and $\bb{Q}_{q,t} [X]_{(m-1)}$, the same formula also holds for $G(X)$ replaced with $\pi_{m,m-1}(G(X))$. By Lemma~\ref{Lem_AHA_action}, we obtain
\begin{align*}
	T_{m-1}^{-1}\cdots T_i^{-1}\bullet G(X)=t^{-1}\sum_{\mb{a}}c_{\mb{a}}X_1^{a_1}\cdots X_{m-2}^{a_{m-2}}X_m^{a_{m-1}}+X_{m-1}H(X)
\end{align*}
for some $H(X)\in\bb{Q}_{q,t} [X]_{(m)}$. Hence we obtain
\begin{align*}
	t^{m-i}\Pi T_{m-1}^{-1}\cdots T_i^{-1}\bullet G(X)=t^{m-i-1}\sum_{\mb{a}}c_{\mb{a}}X_2^{a_1}\cdots X_{m-1}^{a_{m-2}}(q^{-1}X_1)^{a_{m-1}}+X_{m}(\Pi\bullet H(X)).
\end{align*}
On the other hand, we have
\begin{align*}
	t^{m-1-i}\Pi T_{m-2}^{-1}\cdots T_{i}^{-1}\bullet\pi_{m,m-1}(G(X)))=t^{m-i-1}\sum_{\mb{a}}c_{\mb{a}}X_2^{a_1}\cdots X_{m-2}^{a_{m-2}}(q^{-1}X_1)^{a_{m-1}}.
\end{align*}
Therefore, we obtain $\pi_{m,m-1}(Y_i\bullet G(X))=Y_{i}\bullet\pi_{m,m-1}(G(X))$ as desired.
\end{proof}

\begin{corollary}\label{Cor_q_stability}
For any $F(Y)\in\bb{Q}_{q,t}[Y]_{(m)}$ and $0<m'<m$, we have
\begin{align*}
	\msf{q}_{(m')}(\pi_{m,m'}(F(Y)))=\pi_{m,m'}(\msf{q}_{(m)}(F(Y))).
\end{align*}
In particular, we obtain a well-defined $\bb{Q}_{q,t}$-linear isomorphism
\begin{align*}
	\msf{q}:\bb{Q}_{q,t}[Y]_{(\infty)}\xrightarrow{\sim}\bb{Q}_{q,t}[X]_{(\infty)}.
\end{align*}
\end{corollary}

\begin{proof}
	This is the case of $G(X)=1$ in Proposition~\ref{Prop_stability}.
\end{proof}

\begin{corollary}\label{Cor_spmult}
For each $F,G\in\bb{Q}_{q,t} [X]_{(m)}$ and $0<m'<m$, we have
\begin{align*}
	\pi_{m,m'}(F\star G)=\pi_{m,m'}(F)\star \pi_{m,m'}(G).
\end{align*}
In particular, we obtain a well-defined multiplication $\star$ on $\bb{Q}_{q,t}[X]_{(\infty)}$ and $\Lambda_{q,t}$.
\end{corollary}

\begin{proof}
By Corollary~\ref{Cor_q_stability}, we obtain
\begin{align*}
	\pi_{m,m'}(F\star G)&=\msf{q}_{(m')}\left(\pi_{m,m'}(\msf{q}_{(m)}^{-1}(F))\cdot\pi_{m,m'}(\msf{q}_{(m)}^{-1}(G))\right)\\
	&=\msf{q}_{(m')}\left(\msf{q}_{(m')}^{-1}(\pi_{m,m'}(F))\cdot\msf{q}_{(m')}^{-1}(\pi_{m,m'}(G))\right)\\
	&=\pi_{m,m'}(F)\star \pi_{m,m'}(G).
\end{align*}
\end{proof}

\subsection{Pieri rule}

In this section, we prove the simplest case of the Pieri type formula for the quantum multiplication by $e_1(X)$ on $e_{r}(X)$. In this section, we fix large enough $m>r$ and $1\leq a\leq m$, we write 
\begin{align*}
	e^{(a)}_r&\coloneqq e_r(X_1,\ldots,X_a),\\
	e'^{(a)}_r&\coloneqq e_r(X_a,\ldots,X_m)
\end{align*}
for simplicity. We understand that $e^{(a)}_r=0$ for $r>a$ and $e'^{(a)}_{r}=0$ for $a>m$. 

\begin{lemma}\label{Lem_partial_Pieri}
For any $1\leq a\leq m$, we have an equality 
\begin{align*}
	&(1+T_1+T_2T_1+\cdots+T_{a-1}T_{a-2}\cdots T_1)\Pi\bullet e_r(X_1,\ldots,X_m)\\
	&=(1-q^{-1})\sum_{k=1}^{a}[k]_te^{(a)}_{k}e'^{(a+1)}_{r-k+1}+q^{-1}\sum_{k=0}^{a}e^{(a)}_1e^{(a)}_ke'^{(a+1)}_{r-k}
\end{align*}
in $\bb{Q}_{q,t}[X]_{(m)}$.
\end{lemma}

\begin{proof}
We show this equality by induction on $1\leq a\leq m$. When $a=1$, we calculate
\begin{align*}
	\Pi\bullet e_r(X_1,\ldots,X_m)&=\Pi\bullet\left(e_r(X_1,\ldots,X_{m-1})+X_me_{r-1}(X_1,\ldots,X_{m-1})\right)\\
 &=X_1e_r(X_2,\ldots,X_m)+q^{-1}X_1^2e_{r-1}(X_2,\ldots,X_m)\\
	&=e^{(1)}_1e'^{(2)}_r+q^{-1}e^{(1)}_1e^{(1)}_1e'^{(2)}_{r-1}.
\end{align*}

Assume that the equality holds up to $a$ and show it for $a+1$. By subtracting the equality for $a$ by the equality for $a-1$, we obtain
\begin{align*}
&T_{a-1}T_{a-2}\cdots T_1\Pi\bullet e_r(X_1,\ldots,X_m)\\
&=(1-q^{-1})\sum_{k=1}^{a}[k]_t\left(e^{(a)}_{k}e'^{(a+1)}_{r-k+1}-e^{(a-1)}_{k}e'^{(a)}_{r-k+1}\right)+q^{-1}\sum_{k=0}^{a}\left(e^{(a)}_1e^{(a)}_ke'^{(a+1)}_{r-k}-e^{(a-1)}_1e^{(a-1)}_{k-1}e'^{(a)}_{r-k+1}\right).
\end{align*}
where we use $e_r=0$ for $r<0$.
By using 
\begin{align*}
	e^{(a)}_k&=e^{(a-1)}_k+X_ae^{(a-1)}_{k-1},\\
	e'^{(a+1)}_{r-k+1}&=e'^{(a+2)}_{r-k+1}+X_{a+1}e'^{(a+2)}_{r-k},\\
	e'^{(a)}_{r-k+1}&=e'^{(a+2)}_{r-k+1}+(X_a+X_{a+1})e'^{(a+2)}_{r-k}+X_aX_{a+1}e'^{(a+2)}_{r-k-1},
\end{align*}
we calculate
\begin{align*}
&e^{(a)}_{k}e'^{(a+1)}_{r-k+1}-e^{(a-1)}_{k}e'^{(a)}_{r-k+1}\\
&=(e^{(a-1)}_k+X_ae^{(a-1)}_{k-1})(e'^{(a+2)}_{r-k+1}+X_{a+1}e'^{(a+2)}_{r-k})-e^{(a-1)}_{k}(e'^{(a+2)}_{r-k+1}+(X_a+X_{a+1})e'^{(a+2)}_{r-k}+X_aX_{a+1}e'^{(a+2)}_{r-k-1})\\
&=X_a\left(e^{(a-1)}_{k-1}e'^{(a+2)}_{r-k+1}-e^{(a-1)}_{k}e'^{(a+2)}_{r-k}\right)+X_aX_{a+1}\left(e^{(a-1)}_{k-1}e'^{(a+2)}_{r-k}-e^{(a-1)}_{k}e'^{(a+2)}_{r-k-1}\right)
\end{align*}
and
\begin{align*}
&\sum_{k=0}^{a}\left(e^{(a)}_1e^{(a)}_ke'^{(a+1)}_{r-k}-e^{(a-1)}_1e^{(a-1)}_{k-1}e'^{(a)}_{r-k+1}\right)\\
&=\sum_{k=0}^{a}\left((e^{(a-1)}_1+X_a)(e^{(a-1)}_k+X_ae^{(a-1)}_{k-1})(e'^{(a+2)}_{r-k}+X_{a+1}e'^{(a+2)}_{r-k-1})\right.\\
&\hspace{1em}\left.-e^{(a-1)}_1e^{(a-1)}_{k-1}(e'^{(a+2)}_{r-k+1}+(X_a+X_{a+1})e'^{(a+2)}_{r-k}+X_aX_{a+1}e'^{(a+2)}_{r-k-1})\right)\\
&=\sum_{k=0}^{a}\left(e^{(a-1)}_1e^{(a-1)}_{k}e'^{(a+2)}_{r-k}-e^{(a-1)}_1e^{(a-1)}_{k-1}e'^{(a+2)}_{r-k+1}\right)+X_{a+1}\sum_{k=0}^{a}\left(e^{(a-1)}_1e^{(a-1)}_{k}e'^{(a+2)}_{r-k-1}-e^{(a-1)}_1e^{(a-1)}_{k-1}e'^{(a+2)}_{r-k}\right)\\
&\hspace{1em}+\sum_{k=0}^{a}\left(X_ae^{(a-1)}_{k}e'^{(a+2)}_{r-k}+X_a^2e^{(a-1)}_{k-1}e'^{(a+2)}_{r-k}+X_aX_{a+1}e^{(a-1)}_{k}e'^{(a+2)}_{r-k-1}+X_a^2X_{a+1}e^{(a-1)}_{k-1}e'^{(a+2)}_{r-k-1}\right)\\
&=\sum_{k=0}^{a}\left(X_ae^{(a-1)}_{k}e'^{(a+2)}_{r-k}+X_a^2e^{(a-1)}_{k-1}e'^{(a+2)}_{r-k}+X_aX_{a+1}e^{(a-1)}_{k}e'^{(a+2)}_{r-k-1}+X_a^2X_{a+1}e^{(a-1)}_{k-1}e'^{(a+2)}_{r-k-1}\right).
\end{align*}
Therefore, we obtain
\begin{align*}
&T_{a-1}T_{a-2}\cdots T_1\Pi\bullet e_r(X_1,\ldots,X_m)\\
&=(1-q^{-1})\sum_{k=1}^{a}[k]_t\left(X_a\left(e^{(a-1)}_{k-1}e'^{(a+2)}_{r-k+1}-e^{(a-1)}_{k}e'^{(a+2)}_{r-k}\right)+X_aX_{a+1}\left(e^{(a-1)}_{k-1}e'^{(a+2)}_{r-k}-e^{(a-1)}_{k}e'^{(a+2)}_{r-k-1}\right)\right)\\
&\hspace{1em}+q^{-1}\sum_{k=0}^{a}\left(X_ae^{(a-1)}_{k}e'^{(a+2)}_{r-k}+X_a^2e^{(a-1)}_{k-1}e'^{(a+2)}_{r-k}+X_aX_{a+1}e^{(a-1)}_{k}e'^{(a+2)}_{r-k-1}+X_a^2X_{a+1}e^{(a-1)}_{k-1}e'^{(a+2)}_{r-k-1}\right)\\
&=(1-q^{-1})\sum_{k=0}^{a}t^{k}X_ae^{(a-1)}_{k}e'^{(a+1)}_{r-k}+q^{-1}\sum_{k=0}^{a}X_ae^{(a)}_{k}e'^{(a+1)}_{r-k}.
\end{align*}
In order to prove the case $a+1$, we need to check this equality for $a+1$, i.e.,
\begin{align*}
&T_{a}T_{a-1}\cdots T_1\Pi\bullet e_r(X_1,\ldots,X_m)\\
&=(1-q^{-1})\sum_{k=0}^{a+1}t^{k}X_{a+1}e^{(a)}_{k}e'^{(a+2)}_{r-k}+q^{-1}\sum_{k=0}^{a+1}X_{a+1}e^{(a+1)}_{k}e'^{(a+2)}_{r-k}.
\end{align*}

By Lemma~\ref{Lem_AHA_action}, we obtain 
\begin{align*}
	T_a\bullet X_a&=X_{a+1},\\
	T_a\bullet X_aX_{a+1}&=tX_aX_{a+1},\\
	T_a\bullet X_a^2&=X_{a+1}^2-(t-1)X_aX_{a+1},\\
	T_a\bullet X_a^2X_{a+1}&=X_aX_{a+1}^2.
\end{align*}
Therefore, we calculate
\begin{align*}
&T_{a}T_{a-1}\cdots T_1\Pi\bullet e_r(X_1,\ldots,X_m)\\
&=(1-q^{-1})\sum_{k=1}^{a}[k]_t\left(X_{a+1}\left(e^{(a-1)}_{k-1}e'^{(a+2)}_{r-k+1}-e^{(a-1)}_{k}e'^{(a+2)}_{r-k}\right)+tX_aX_{a+1}\left(e^{(a-1)}_{k-1}e'^{(a+2)}_{r-k}-e^{(a-1)}_{k}e'^{(a+2)}_{r-k-1}\right)\right)\\
&\hspace{1em}+q^{-1}\sum_{k=0}^{a}\left(X_{a+1}e^{(a-1)}_{k}e'^{(a+2)}_{r-k}+(X_{a+1}^2+X_aX_{a+1})e^{(a-1)}_{k-1}e'^{(a+2)}_{r-k}+X_aX_{a+1}^2e^{(a-1)}_{k-1}e'^{(a+2)}_{r-k-1}\right)\\
&=(1-q^{-1})\sum_{k=0}^{a}\left(t^{k}X_{a+1}e^{(a-1)}_{k}e'^{(a+2)}_{r-k}+t^{k+1}X_aX_{a+1}e^{(a-1)}_{k}e'^{(a+2)}_{r-k-1}\right)\\
&\hspace{1em}+q^{-1}\sum_{k=0}^{a+1}\left(X_{a+1}e^{(a-1)}_{k}e'^{(a+2)}_{r-k}+(X_{a+1}^2+X_aX_{a+1})e^{(a-1)}_{k-1}e'^{(a+2)}_{r-k}+X_aX_{a+1}^2e^{(a-1)}_{k-2}e'^{(a+2)}_{r-k}\right)\\
&=(1-q^{-1})\sum_{k=0}^{a+1}t^{k}X_{a+1}e^{(a)}_{k}e'^{(a+2)}_{r-k}+q^{-1}\sum_{k=0}^{a+1}X_{a+1}e^{(a+1)}_{k}e'^{(a+2)}_{r-k}.	
\end{align*}
This completes the proof of the Lemma.
\end{proof}

\begin{thm}[The Pieri rule for the quantum multiplication]\label{Thm_qtPieri_en}
For any $r \in \bb{Z}_{\geq0}$, we have
\begin{align*}
	e_1(X)\star e_r(X)=(1-q^{-1})[r+1]_te_{r+1}(X)+q^{-1}e_1(X)e_r(X).
\end{align*}
\end{thm}

\begin{proof}
This follows from Lemma~\ref{Lem_partial_Pieri} when $a=m$ and 
\begin{align*}
	Y_{i}\bullet e_r(X_1,\ldots,X_m)=T_{i-1}\cdots T_1\Pi\bullet e_r(X_1,\ldots,X_m).
\end{align*}
for $i=1,\ldots,m$.
\end{proof}

\section{$(q,t)$-chromatic symmetric functions}

In this section, we define $(q,t)$-chromatic symmetric functions and check some basic properties.

\subsection{Partial symmetrizers}

We first construct certain elements in $\msc{H}_m$ and study its properties.

\begin{dfn}\label{Def_ps}
We define $\msf{S}^{(m)}_{i,e}$ in $\msc{H}_m$ for $i\in\bb{Z}/m\bb{Z}$ and $e\in\bb{Z}$ by 
\begin{align*}
	\msf{S}^{(m)}_{i,e}\coloneqq\begin{cases}
		1+T_i^{-1}+T_{i+1}^{-1}T_i^{-1}+\cdots+T_{m+e-1}^{-1}\cdots T_i^{-1} &\mbox{ if }e\geq i-m,\\
		0 &\mbox{ if }e<i-m.
	\end{cases}
\end{align*}
We also introduce $\widehat{\msf{S}}^{(m)}_a\in\msc{H}_m$ for $a\in\bb{Z}$ by
\begin{align*}
	\widehat{\msf{S}}^{(m)}_{a}\coloneqq \msf{S}^{(m)}_{1,a-m+1}\Pi=\begin{cases}
		(1+T_1^{-1}+T_{2}^{-1}T_1^{-1}+\cdots+T_{a}^{-1}\cdots T_1^{-1})\Pi &\mbox{ if }a\geq 0,\\
		0&\mbox{ if }a<0. 
	\end{cases}
\end{align*}
\end{dfn}

For example, we have $\msf{S}^{(m)}_{i,i-m}=1$ and $\widehat{\msf{S}}^{(m)}_{0}=\Pi$. In this paper, we only consider the case $e<i$ for $\msf{S}^{(m)}_{i,e}$ and $a<m$ for $\widehat{\msf{S}}^{(m)}_{a}$.

\begin{lemma}\label{Lem_Sinv}
Let $0< a<m$ and $F(X)\in\mathbb Q_{q,t} [X^{\pm1}]_{(m)}$. The element $\widehat{\msf{S}}^{(m)}_{a}\bullet F(X) \in\mathbb Q_{q,t} [X^{\pm1}]_{(m)}$ is $s_i$-invariant if
\begin{enumerate}
	\item $i=1$ and $F(X)$ is invariant under $s_{1}$,
	\item $1 < i < a$ and $F(X)$ is invariant under $s_{i-1}$ and $s_i$, or
	\item $i=a$ or $a+1<i<m$, and $F(X)$ is invariant under $s_{i-1}$.
\end{enumerate}
\end{lemma}

\begin{proof}
Recall that $F(X)$ is $s_i$-invariant if and only if $T_i^{-1}\bullet F(X)=t^{-1}F(X)$. 

Since $a>0$, we have
\begin{align*}
	T_1^{-1}\widehat{\msf{S}}^{(m)}_{a}&=(T_1^{-1}+T_1^{-2}+T_1^{-1}T_2^{-1}T_1^{-1}+\cdots+T_a^{-1}\cdots T_3^{-1}T_1^{-1}T_2^{-1}T_1^{-1})\Pi\\
	&=t^{-1}(1+T_1^{-1})\Pi+(T_2^{-1}T_1^{-1}+\cdots+T_a^{-1}\cdots T_1^{-1})T_2^{-1}\Pi\\
	&=t^{-1}(1+T_1^{-1})\Pi+(T_2^{-1}T_1^{-1}+\cdots+T_a^{-1}\cdots T_1^{-1}) \Pi T_1^{-1}
\end{align*}
by Definition~\ref{def:AHA}. If we have $T_1^{-1}\bullet F(X)=t^{-1}F(X)$, then we obtain $T_1^{-1}\widehat{\msf{S}}^{(m)}_{a}\bullet F(X)=t^{-1}\widehat{\msf{S}}^{(m)}_{a}\bullet F(X)$. This is the first case.

When $1<i<a$, we have
\begin{align*}
	T_i^{-1}\widehat{\msf{S}}^{(m)}_{a}&=(T_i^{-1}+\cdots+T_i^{-1}T_{i-1}^{-1}\cdots T_1^{-1}+T_i^{-2}T_{i-1}^{-1}\cdots T_1^{-1}+\cdots +T_a^{-1}\cdots T_i^{-1}T_{i+1}^{-1}T_i^{-1}\cdots T_1^{-1})\Pi\\
	&=(1+T_1^{-1}+\cdots+T_{i-2}^{-1}\cdots T_1^{-1})\Pi T_{i-1}^{-1}+t^{-1}(T_{i-1}^{-1}\cdots T_1^{-1}+T_{i}^{-1}\cdots T_1^{-1})\Pi\\
	&\hspace{1em}+(T_{i+1}^{-1}\cdots T_1^{-1}+\cdots T_a^{-1}\cdots T_1^{-1})\Pi T_{i}^{-1},
\end{align*}
by Definition~\ref{def:AHA}. If we have $T_{i-1}^{-1}\bullet F(X)=T_i^{-1}\bullet F(X)=t^{-1}F(X)$, then we obtain $T_i^{-1}\widehat{\msf{S}}^{(m)}_{a}\bullet F(X)=t^{-1}\widehat{\msf{S}}^{(m)}_{a} F(X)$. This is the second case.

When $i=a$, we have
\begin{align*}
	T_a^{-1}\widehat{\msf{S}}^{(m)}_{a}=(1+T_1^{-1}+\cdots+T_{a-2}^{-1}\cdots T_1^{-1})\Pi T_{a-1}^{-1}+t^{-1}(T_{a-1}^{-1}\cdots T_1^{-1}+T_{a}^{-1}\cdots T_1^{-1})\Pi
\end{align*}
by Definition~\ref{def:AHA}. If we have $T_{a-1}^{-1}\bullet F(X)=t^{-1}F(X)$, then we obtain $T_a^{-1}\widehat{\msf{S}}^{(m)}_{a}\bullet F(X)=t^{-1}\widehat{\msf{S}}^{(m)}_{a}\bullet F(X)$. If we have $i>a+1$, then we have $T_i^{-1}\widehat{\msf{S}}^{(m)}_{a}=\widehat{\msf{S}}^{(m)}_{a}T_{i-1}^{-1}$ and hence the same implication holds. This is the third case.

This completes the proof.
\end{proof}

In particular, we obtain the following criterion for $\widehat{\msf{S}}^{(m)}_{m-1}\bullet F(X)$ to be symmetric.

\begin{corollary}
If $F\in\mathbb Q_{q,t} [X^{\pm1}]_{(m)}$ is invariant under $s_1,s_2,\ldots,s_{m-2}$, then $\widehat{\msf{S}}^{(m)}_{m-1}\bullet F(X)$ is symmetric.
\end{corollary}

\subsection{$(q,t)$-chromatic symmetric functions}

Now we define a $(q,t)$-analogue of chromatic symmetric functions for unit interval graphs. 

\begin{dfn}
	For $\msf{e}\in\bb{E}_n$ and $m\in\bb{Z}_{>0}$, we define an element $\msf{S}_{\msf{e}}^{(m)}\in\msc{H}_m$ by
\begin{align*}
	\msf{S}_{\msf{e}}^{(m)}\coloneqq t^{n(m-1)}\msf{S}_{1,\msf{e}(1)}^{(m)}\msf{S}_{2,\msf{e}(2)}^{(m)}\cdots \msf{S}_{n,\msf{e}(n)}^{(m)}\Pi^n.
\end{align*}
\end{dfn}

We note that by 
\begin{align*}
\msf{S}_{i,e}^{(m)}=\Pi^{i-1}\msf{S}_{1,e-i+1}^{(m)}\Pi^{-i+1},
\end{align*}
we can also write
\begin{align*}
	\msf{S}_{\msf{e}}^{(m)}= t^{n(m-1)}\widehat{\msf{S}}_{m-1-\msf{a}(1)}^{(m)}\widehat{\msf{S}}_{m-1-\msf{a}(2)}^{(m)}\cdots\widehat{\msf{S}}_{m-1-\msf{a}(n)}^{(m)},
\end{align*}
where $\msf{a}\in\bb{A}_n$ is the area sequence corresponding to $\msf{e}\in\bb{E}_n$, i.e., $\msf{a}(i)=i-1-\msf{e}(i)$. We will show in Theorem~\ref{Thm_stability} that every $F=\varprojlim_{m}\left(F^{(m)}\right)\in\Lambda_{q,t}$ gives rise to an element
\begin{equation}
	\msf{S}_{\msf{e}}(F)\coloneqq\varprojlim_{m}\left(\msf{S}_{\msf{e}}^{(m)}\bullet F^{(m)}\right)\in\Lambda_{q,t}.\label{eqn:preStab}
\end{equation}
Assuming the existence of (\ref{eqn:preStab}), we make the following definition:

\begin{dfn}
For a unit interval graph $\Gamma$ corresponding to $\msf{e}\in\bb{E}_n$, we define  
\begin{align*}
	\mb{X}_{\Gamma}(q,t)\coloneqq \msf{S}_{\msf{e}}(1)\in\Lambda_{q,t}
\end{align*}
and call it the $(q,t)$\textit{-chromatic symmetric function} for $\Gamma$. We also define
\begin{align*}
\mb{X}_{\Gamma}^{(m)}(q,t) \coloneqq \msf{S}_{\msf{e}}^{(m)}\bullet 1 
\end{align*}
for each $m \in \bb{Z}_{>0}$ and call it the $m$-truncated $(q,t)$-chromatic symmetric function for $\Gamma$.
\end{dfn}

\begin{exa}\label{ex:small}
When $n=3$ and $\msf{e} = (0,0,1)$, the corresponding area sequence is $\msf{a}=(0,1,1)$ and the corresponding graph $\Gamma$ is a Dynkin diagram of type $A_3$. We have
\begin{align*}
\mb{X}_{\Gamma}^{(2)}(q,t) & = t^3 (1+T_1^{-1})\Pi^3 \bullet1  = q^{-1}t^2 (X_1+X_2)X_1X_2 \equiv q^{-1}t^2 e_{2,1}(X) \mod (X_3,X_4,\ldots)\\
\mb{X}_{\Gamma}^{(3)}(q,t) & = t^6 (1+T_1^{-1}+T_2^{-1} T_1^{-1})\Pi (1+T_1^{-1})\Pi(1+T_1^{-1})\Pi \bullet1\\
& = t^5 (1+T_1^{-1}+T_2^{-1} T_1^{-1}) \Pi (1+T_1^{-1}) \bullet X_1(X_2 + X_3)\\
& = t^4 (1+T_1^{-1}+T_2^{-1} T_1^{-1}) \Pi \bullet\left( X_1 X_2 + X_2 X_3 + X_1 X_3 + t X_1 X_2\right)\\
& \equiv q^{-1}t^2 ( e_{2,1}(X) + (1+t+t^2)(-1+q+qt) e_{3}(X) ) \mod (X_4,X_5,\ldots) \\
\mb{X}_{\Gamma}^{(4)}(q,t) & = t^9  (1+T_1^{-1}+T_2^{-1} T_1^{-1}+T_3^{-1}T_2^{-1} T_1^{-1}) \Pi (1+T_1^{-1}+T_2^{-1}T_1^{-1})\Pi(1+T_1^{-1}+T_{2}^{-1}T_1^{-1})\Pi \bullet1\\
&  \equiv q^{-1}t^2 (e_{2,1}(X) + (1+t+t^2)(-1+q+qt) e_{3}(X)) \mod (X_5,X_6,\ldots)
\end{align*}
We observe that the resulting polynomials are stable, and also not $e$-positive in a naive sense.
\end{exa}

\begin{remark}
By our convention on the operator $\msf{S}^{(m)}_{i,e}$ in Definition~\ref{Def_ps}, the $(q,t)$-chromatic symmetric function $\mb{X}_{\Gamma}^{(m)}(q,t)$ for $\Gamma$ corresponding to $\msf{e}\in\bb{E}_n$ automatically vanishes whenever $\msf{e}(i) < i - m$ for some $1 \leq i \leq m$. 
\end{remark}

\subsection{Symmetricity}

In this section, we prove that $\mb{X}_{\Gamma}^{(m)}(q,t)$ is symmetric for any unit interval graph $\Gamma$. 

\begin{lemma}\label{Lem_Sesym}
For any $\msf{e}\in\bb{E}_n$ and $F(X)\in\bb{Q}_{q,t}[X^{\pm1}]_{(m)}$ which is invariant under $s_1,\ldots, s_{m-2}$, $\msf{S}^{(m)}_{\msf{e}}\bullet F(X)$ is symmetric.
\end{lemma}

\begin{proof}
Let $\msf{a}\in\bb{A}_n$ be the area sequence corresponding to $\msf{e}\in\bb{E}_n$. We inductively construct a nondecreasing sequence of integers $k_1,\ldots,k_{n+1}$ by $k_1=0$ and 
\begin{align*}
	k_{j+1}=\begin{cases}
		k_{j}&\mbox{ if }\msf{a}(j)<k_j,\\
		k_{j}+1&\mbox{ if }\msf{a}(j)=k_{j},
	\end{cases}
\end{align*}
for $1\leq j\leq n$. We note that $\msf{a}(j)\leq k_j$ holds for any $j$. Since we should have $\msf{a}(1)=0=k_1$, we have $k_2=1$ and hence $k_{n+1}\geq1$.

We set 
\begin{align}\label{Eqn_Fj}
	F_j(X)\coloneqq\widehat{\msf{S}}^{(m)}_{m-1-\msf{a}(j)}\widehat{\msf{S}}^{(m)}_{m-1-\msf{a}(j+1)}\cdots\widehat{\msf{S}}^{(m)}_{m-1-\msf{a}(n)}\bullet F(X)
\end{align}
for each $1\leq j\leq n$ and $F_{n+1}(X)\coloneqq F(X)$. We examine the symmetry of 
\begin{align*}
	F_j(X)=\widehat{\msf{S}}^{(m)}_{m-1-\msf{a}(j)}\bullet F_{j+1}(X)
\end{align*}
in accordance with the values of $k_{j+1}$. By assumption and $k_{n+1}\geq1$, $F_{n+1}(X)$ is invariant under $s_1,\ldots,s_{m-1-k_{n+1}}$. By descending induction on $j$, we show that $F_j(X)$ is invariant under $s_1,\ldots,s_{m-1-k_{j}}$.

Assume that $F_{j+1}(X)$ is invariant under $s_1,\ldots,s_{m-1-k_{j+1}}$ and $k_{j+1}<m-1$. In this case, we note that 
\begin{align*}
	m-1-\msf{a}(j)\geq m-1-k_j\geq  m-1-k_{j+1}>0.
\end{align*}
Since $k_{j+1}<m-1$, $F_{j+1}(X)$ is $s_1$-invariant. Hence by Lemma~\ref{Lem_Sinv}, $F_j(X)$ is $s_1$-invariant. 

If $1<i<m-1-k_j\leq m-1-\msf{a}(j)$, then we have $1<i\leq m-1-k_{j+1}$ by the construction of $k_{j+1}$ and hence $F_{j+1}(X)$ is invariant under $s_{i-1}$ and $s_i$. By Lemma~\ref{Lem_Sinv}, $F_j(X)$ is $s_i$-invariant.

If $i=m-1-k_j<m-1-\msf{a}(j)$, then we have $k_{j+1}=k_{j}$ and hence  $F_{j+1}(X)$ is invariant under $s_{i-1}$ and $s_i$. By Lemma~\ref{Lem_Sinv}, $F_j(X)$ is $s_i$-invariant.

If $i=m-1-k_j=m-1-\msf{a}(j)$, then we have $k_{j+1}=k_{j}+1$ and hence  $F_{j+1}(X)$ is $s_{i-1}$-invariant. By Lemma~\ref{Lem_Sinv}, $F_j(X)$ is $s_i$-invariant.

Therefore, $F_j(X)$ is invariant under $s_1,\ldots,s_{m-1-k_j}$ in this case.

Assume now that $k_{j+1}\geq m-1$. If $k_j\geq m-1$, we have nothing to prove. Hence we may assume $k_{j+1}=m-1$ and $k_j=m-2$. In this case, we have $\msf{a}(j)=m-2$ and hence 
\begin{align*}
	F_j(X)=\widehat{\msf{S}}^{(m)}_{1}\bullet F_{j+1}(X)=(1+T_1^{-1})\Pi\bullet F_{j+1}(X)
\end{align*}
is $s_1$-invariant without any condition on $F_{j+1}(X)$.  

Therefore, $t^{n(m-n)}F_1(X)=\msf{S}^{(m)}_{\msf{e}}\bullet F(X)$ is invariant under $s_1,\ldots,s_{m-1-k_1}=s_{m-1}$ and hence symmetric. This completes the proof.
\end{proof}

\begin{corollary}\label{Cor_symm}
For any unit interval graph $\Gamma$, we have $\mb{X}_{\Gamma}^{(m)}(q,t)\in\Lambda_{q,t}^{(m)}$.
\end{corollary}

\begin{proof}
	This follows from Lemma~\ref{Lem_Sesym} for $F(X)=1$.
\end{proof}

As we discussed in the introduction, we can obtain the multiplicativity of the $m$-truncated $(q,t)$-chromatic symmetric functions from this result.

\begin{corollary}\label{Cor_fact}
For any two unit interval graphs $\Gamma$ and $\Gamma'$, we have
\begin{align*}
\mb{X}_{\Gamma\cup\Gamma'}^{(m)}(q,t)=\mb{X}_{\Gamma}^{(m)}(q,t) \star \mb{X}_{\Gamma'}^{(m)}(q,t).
\end{align*}
\end{corollary}

\begin{proof}
Let $\msf{e}\in\bb{E}_n$ (resp. $\msf{e}'\in\bb{E}_{n'}$) correspond to unit interval graph $\Gamma$ (resp. $\Gamma'$). By the description of concatenation in Definition~\ref{Def_concat}, we obtain
\begin{align*}
	\msf{S}_{\msf{e}\cup\msf{e}'}^{(m)}=\msf{S}_{\msf{e}}^{(m)}\msf{S}_{\msf{e}'}^{(m)}
\end{align*}
as elements in $\msc{H}_m$. By Lemma~\ref{Lem_star_commute} and Corollary~\ref{Cor_symm}, we obtain
\begin{align*}
\mb{X}_{\Gamma\cup\Gamma'}^{(m)}(q,t)=\msf{S}_{\msf{e}\cup\msf{e}'}^{(m)}\bullet 1	=\msf{S}_{\msf{e}}^{(m)}\bullet(1\star \mb{X}_{\Gamma'}^{(m)}(q,t))=\mb{X}_{\Gamma}^{(m)}(q,t)\star \mb{X}_{\Gamma'}^{(m)}(q,t)
\end{align*}
as required.
\end{proof}

\subsection{Complete graphs}

In this section, we compute the $(q,t)$-chromatic symmetric functions for complete graphs, i.e., when $\msf{e}=\msf{e}_n\in\bb{E}_n$ in the notation of Example~\ref{Exa_complete_graph}.

\begin{lemma}\label{Lem_Selem}
For each $0\leq a\leq b<m$ and $r\in\bb{Z}_{\geq 0}$, we have
\begin{align}
	\widehat{\msf{S}}^{(m)}_{a}\bullet e_r(X_1,\ldots,X_b)=t^{-a}\sum_{k=0}^{a}[k+1]_t e_{k+1}(X_1,\ldots,X_{a+1})e_{r-k}(X_{a+2},\ldots,X_{b+1})\label{eqn:Sae}
\end{align}
in $\bb{Q}_{q,t} [X^{\pm1}]_{(m)}$. Here, we understand that $e_{k}(X_{a+2},\ldots,X_{b+1})=\delta_{k,0}$ when $a=b$.
\end{lemma}

\begin{proof}
For $a=0$, we have
\begin{align*}
	\widehat{\msf{S}}^{(m)}_{0}\bullet e_r(X_1,\ldots,X_b)=\Pi\bullet e_r(X_1,\ldots,X_b)=X_1e_r(X_2,\ldots,X_{b+1}).
\end{align*}
This is the (\ref{eqn:Sae}) for $a=0$, and hence the assertion holds in this case. 


We prove the assertion by induction on $a$. In order to shorten the expressions, we write $e^{(a)}_k=e_k(X_1,\ldots,X_a)$ and $e'^{(a)}_{k}=e_k(X_{a},\ldots,X_{b+1})$ in the rest of this proof. By induction, it is enough to show that the equation (\ref{eqn:Sae}) for $a-1$ subtracted from the equation (\ref{eqn:Sae}) for $a$ holds, i.e.,
\begin{align}\nonumber
T_a^{-1}\cdots T_1^{-1}\Pi\bullet e_r(X_1,\ldots,X_b)&=t^{-a}\sum_{k=0}^{a}[k+1]_t e^{(a+1)}_{k+1}e'^{(a+2)}_{r-k}-t^{-a+1}\sum_{k=0}^{a-1}[k+1]_t e^{(a)}_{k+1}e'^{(a+1)}_{r-k}\\\label{eqn:Seind}
&=t^{-a}\sum_{k=0}^{a}[k+1]_t \left(e^{(a+1)}_{k+1}e'^{(a+2)}_{r-k}-te^{(a)}_{k+1}e'^{(a+1)}_{r-k}\right).
\end{align}
We assume (\ref{eqn:Seind}) for $a$ and prove it for $a+1$. We may assume $a<b$. Since we have
\begin{align*}
e^{(a+1)}_{k+1}e'^{(a+2)}_{r-k}-te^{(a)}_{k+1}e'^{(a+1)}_{r-k}&=\left(e^{(a)}_{k+1}+X_{a+1}e^{(a)}_{k}\right)\left(e'^{(a+3)}_{r-k}+X_{a+2}e'^{(a+3)}_{r-k-1}\right)\\
&\hspace{1em}-te^{(a)}_{k+1}\left(e'^{(a+3)}_{r-k}+(X_{a+1}+X_{a+2})e'^{(a+3)}_{r-k-1}+X_{a+1}X_{a+2}e'^{(a+3)}_{r-k-2}\right)\\
&=(1-t)e^{(a)}_{k+1}e'^{(a+3)}_{r-k}+X_{a+1}\left(e^{(a)}_{k}e'^{(a+3)}_{r-k}-te^{(a)}_{k+1}e'^{(a+3)}_{r-k-1}\right)\\
&\hspace{1em}+(1-t)X_{a+2}e^{(a)}_{k+1}e'^{(a+3)}_{r-k-1}+X_{a+1}X_{a+2}\left(e^{(a)}_{k}e'^{(a+3)}_{r-k-1}-te^{(a)}_{k+1}e'^{(a+3)}_{r-k-2}\right),
\end{align*}
we can rewrite (\ref{eqn:Seind}) for $a$ as
\begin{align}\label{eqn:Seindr}
&T_a^{-1}\cdots T_1^{-1}\Pi\bullet e_r(X_1,\ldots,X_b)\\\nonumber
&=t^{-a}(1-t)\sum_{k=0}^{a-1}[k+1]_te^{(a)}_{k+1}e'^{(a+3)}_{r-k}+t^{-a}X_{a+1}\sum_{k=0}^{a}e^{(a)}_{k}e'^{(a+3)}_{r-k}\\\nonumber
&\hspace{1em}+t^{-a}(1-t)X_{a+2}\sum_{k=0}^{a-1}[k+1]_te^{(a)}_{k+1}e'^{(a+3)}_{r-k-1}+t^{-a}X_{a+1}X_{a+2}\sum_{k=0}^{a}e^{(a)}_{k}e'^{(a+3)}_{r-k-1},
\end{align}
where we have used $e'^{(a)}_{a+1} = 0$ in the first and the third sums, and used $[k+1]_t - t [k]_t = 1$ in the second and the fourth sums. Note that $e$ and $e'$ in the RHS of (\ref{eqn:Seindr}) does not involve $X_{a+1},X_{a+2}$, and hence $T_{a+1}^{-1}$ commutes with them. Therefore, we have
\begin{align}\label{eqn:SaeLHS}
&T_{a+1}^{-1}\cdots T_1^{-1}\Pi\bullet e_r(X_1,\ldots,X_b)\\\nonumber
&=t^{-a-1}(1-t)\sum_{k=0}^{a-1}[k+1]_te^{(a)}_{k+1}e'^{(a+3)}_{r-k}+t^{-a-1}\left(X_{a+2}+(1-t)X_{a+1}\right)\sum_{k=0}^{a}e^{(a)}_{k}e'^{(a+3)}_{r-k}\\\nonumber
&\hspace{1em}+t^{-a-1}(1-t)X_{a+1}\sum_{k=0}^{a-1}t[k+1]_te^{(a)}_{k+1}e'^{(a+3)}_{r-k-1}+t^{-a-1}X_{a+1}X_{a+2}\sum_{k=0}^{a}e^{(a)}_{k}e'^{(a+3)}_{r-k-1}\\\nonumber
&=t^{-a-1}(1-t)\sum_{k=0}^{a-1}[k+1]_te^{(a)}_{k+1}e'^{(a+3)}_{r-k}+t^{-a-1}X_{a+2}\sum_{k=0}^{a}e^{(a)}_{k}e'^{(a+3)}_{r-k}\\\nonumber
&\hspace{1em}+t^{-a-1}(1-t)X_{a+1}\sum_{k=0}^{a}[k+1]_te^{(a)}_{k}e'^{(a+3)}_{r-k}+t^{-a-1}X_{a+1}X_{a+2}\sum_{k=0}^{a}e^{(a)}_{k}e'^{(a+3)}_{r-k-1}.
\end{align}

On the other hand, since we have
\begin{align*}
e^{(a+2)}_{k+1}e'^{(a+3)}_{r-k}-te^{(a+1)}_{k+1}e'^{(a+2)}_{r-k}&=\left(e^{(a)}_{k+1}+(X_{a+1}+X_{a+2})e^{(a)}_{k}+X_{a+1}X_{a+2}e^{(a)}_{k-1}\right)e'^{(a+3)}_{r-k}\\
&\hspace{1em}-t\left(e^{(a)}_{k+1}+X_{a+1}e^{(a)}_{k}\right)\left(e'^{(a+3)}_{r-k}+X_{a+2}e'^{(a+3)}_{r-k-1}\right)\\
&=(1-t)e^{(a)}_{k+1}e'^{(a+3)}_{r-k}+X_{a+2}\left(e^{(a)}_{k}e'^{(a+3)}_{r-k}-te^{(a)}_{k+1}e'^{(a+3)}_{r-k-1}\right)\\
&\hspace{1em}+(1-t)X_{a+1}e^{(a)}_{k}e'^{(a+3)}_{r-k}+X_{a+1}X_{a+2}\left(e^{(a)}_{k-1}e'^{(a+3)}_{r-k}-te^{(a)}_{k}e'^{(a+3)}_{r-k-1}\right),
\end{align*}
we obtain 
\begin{align}\label{eqn:SaeRHS}
&t^{-a-1}\sum_{k=0}^{a+1}[k+1]_t \left(e^{(a+2)}_{k+1}e'^{(a+3)}_{r-k}-te^{(a+1)}_{k+1}e'^{(a+2)}_{r-k}\right)\\\nonumber
&=t^{-a-1}(1-t)\sum_{k=0}^{a-1}[k+1]_te^{(a)}_{k+1}e'^{(a+3)}_{r-k}+t^{-a-1}X_{a+2}\sum_{k=0}^{a}e^{(a)}_{k}e'^{(a+3)}_{r-k}\\\nonumber
&\hspace{1em}+t^{-a-1}(1-t)X_{a+1}\sum_{k=0}^{a}[k+1]_te^{(a)}_{k}e'^{(a+3)}_{r-k}+t^{-a-1}X_{a+1}X_{a+2}\sum_{k=0}^{a}e^{(a)}_{k}e'^{(a+3)}_{r-k-1}
\end{align}
where we used $e^{(a)}_{a+1} = e^{(a)}_{a+2} = 0$ to reduce terms at all sums. Identifying the RHS of (\ref{eqn:SaeLHS}) and the RHS of (\ref{eqn:SaeRHS}), the induction proceeds. This completes the proof.
\end{proof}

\begin{corollary}\label{Cor_Sa_elem}
For any $0\leq a<m$ and $r\in\bb{Z}_{\geq0}$, we have
\begin{align*}
	\widehat{\msf{S}}^{(m)}_{a}\bullet e_r(X_1,\ldots,X_a)=t^{-a}[r+1]_te_{r+1}(X_1,\ldots,X_{a+1}).
\end{align*} 
In particular, we have
\begin{align*}
	\widehat{\msf{S}}^{(m)}_{a}\bullet 1=t^{-a}e_{1}(X_1,\ldots,X_{a+1}).
\end{align*}
\end{corollary}

\begin{proof}
This is a special case $b=a$ of Lemma~\ref{Lem_Selem}.	
\end{proof}

\begin{corollary}\label{Cor_complete}
For any $m,n \in \bb{Z}_{>0}$, we have
\begin{align*}
	\msf{S}^{(m)}_{\msf{e}_n}\bullet1=t^{n(n-1)/2}[n]_t!\cdot e_n(X_1,\ldots,X_m).
\end{align*}
In particular, this will stabilize at $m\rightarrow\infty$ and obtain
\begin{align*}
	\mb{X}_{\Gamma_{\msf{e}_n}}(q,t)=t^{n(n-1)/2}[n]_t!\cdot e_n(X).
\end{align*}
\end{corollary}

\begin{proof}
The area sequence $\msf{a}_n\in\bb{A}_n$ corresponding to $\msf{e}_n\in\bb{E}_n$ is given by $\msf{a}_n(i)=i-1$ for each $i=1,\ldots,n$. 
By using Corollary~\ref{Cor_Sa_elem} repeatedly, we obtain
\begin{align*}
\msf{S}^{(m)}_{\msf{e}_n}\bullet1&=t^{n(m-1)}\widehat{\msf{S}}^{(m)}_{m-1}\widehat{\msf{S}}^{(m)}_{m-2}\cdots\widehat{\msf{S}}^{(m)}_{m-n}\bullet 1\\
&=t^{n(m-1) -(m-n)} \widehat{\msf{S}}^{(m)}_{m-1}\widehat{\msf{S}}^{(m)}_{m-2}\cdots\widehat{\msf{S}}^{(m)}_{m-n+1}\bullet e_1(X_1,\ldots,X_{m-n+1})\\
&=t^{n(m-1) -(m-n)-(m-n+1)} [2]_t\cdot\widehat{\msf{S}}^{(m)}_{m-1}\widehat{\msf{S}}^{(m)}_{m-2}\cdots\widehat{\msf{S}}^{(m)}_{m-n+2}\bullet e_2(X_1,\ldots,X_{m-n+2})\\
&=\cdots\\
&=t^{n(n-1)/2}[n]_t!\cdot e_n(X_1,\ldots,X_m)
\end{align*}
as required, where we used the equality
\begin{align*}
	n(m-1) - \sum_{j=1}^n (m-j) = \frac{n(n-1)}{2}.
\end{align*}
\end{proof}

\section{Modular law}

In this section, we prove that the $m$-truncated $(q,t)$-chromatic symmetric functions satisfies the modular law in the sense of Abreu-Nigro \cite{AN21} recalled in Lemma~\ref{Lem_emodular}. Our main results will follow from the modular law.

\subsection{Preparations}

In this section, we prove two preparatory lemmas needed to prove the modular law.

\begin{lemma}\label{Lem_modI}
For any $a<m-1$, we have
\begin{align*}
	\widehat{\msf{S}}^{(m)}_{a+1}\widehat{\msf{S}}^{(m)}_{a}(1-tT_{m-1}^{-1})=0
\end{align*}
as an element of $\msc{H}_m$.
\end{lemma}

\begin{proof}
If $a<0$, then the equality is trivial as $\widehat{\msf{S}}^{(m)}_{a}=0$. If $0\leq a<m-1$, then the LHS is given by
\begin{align}\nonumber
&(1+T_1^{-1}+\cdots+T_{a+1}^{-1}\cdots T_1^{-1})\Pi(1+T_1^{-1}+\cdots+T_{a}^{-1}\cdots T_1^{-1})\Pi(1-tT_{m-1}^{-1})\\\nonumber
&=\Pi(1+T_0^{-1}+\cdots+T_{a}^{-1}\cdots T_0^{-1})(1+T_1^{-1}+\cdots+T_{a}^{-1}\cdots T_1^{-1})(1-tT_0^{-1})\Pi\\
&=\Pi\sum_{j=-1}^{a}\sum_{k=0}^{a} T_j^{-1}\cdots T_0^{-1}\cdot T_k^{-1}\cdots T_1^{-1}(1-tT_0^{-1})\Pi,\label{eqn:PTTP}
\end{align}
where $j=-1$ or $k=0$ means the corresponding monomials are understood to be $1$. If $0\leq j+1\leq k<m-1$, then we have
\begin{align*}
	T_j^{-1}\cdots T_0^{-1}\cdot T_k^{-1}\cdots T_0^{-1}&=T_j^{-1}\cdots T_1^{-1}T_k^{-1}\cdots T_2^{-1}T_0^{-1}T_1^{-1}T_0^{-1}\\
	&=T_j^{-1}\cdots T_1^{-1}T_k^{-1}\cdots T_2^{-1}T_1^{-1}T_0^{-1}\cdot T_1^{-1}\\
	&=\cdots\\
	&=T_k^{-1}\cdots T_0^{-1}\cdot T_{j+1}^{-1}\cdots T_1^{-1}
\end{align*}
by Definition~\ref{def:AHA}. If $m-1>j\geq k\geq0$, then we have
\begin{align*}
T_j^{-1}\cdots T_0^{-1}\cdot T_k^{-1}\cdots T_0^{-1}&=T_j^{-1}\cdots T_k^{-1}T_{k-1}^{-1}T_k^{-1}\cdots T_0^{-1}\cdot T_{k-1}^{-1}\cdots T_0^{-1}\\
&=T_j^{-1}\cdots T_{k-1}^{-1}T_{k}^{-1}T_{k-1}^{-1}\cdots T_0^{-1}\cdot T_{k-1}^{-1}\cdots T_0^{-1}\\
&= T_{k-1}^{-1}T_j^{-1}\cdots T_0^{-1}\cdot T_{k-1}^{-1}\cdots T_0^{-1}\\
& = \cdots\\
	&= ( T_{k-1}^{-1}\cdots T_0^{-1} ) \cdot T_j^{-1}\cdots T_1^{-1} T_0^{-1} T_0^{-1}\\
	&= (t^{-1}-1) ( T_{k-1}^{-1}\cdots T_0^{-1} ) ( T_j^{-1}\cdots T_0^{-1} ) + t^{-1}  ( T_{k-1}^{-1}\cdots T_0^{-1} ) ( T_j^{-1}\cdots T_1^{-1} )\\
	&=(t^{-1}-1)T_j^{-1}\cdots T_0^{-1}\cdot T_k^{-1}\cdots T_1^{-1}+t^{-1}T_{k-1}^{-1}\cdots T_0^{-1}\cdot T_j^{-1}\cdots T_1^{-1}
\end{align*}
by Definition~\ref{def:AHA} and the case $0\leq j+1\leq k<m-1$ above. Summing them up, we obtain
\begin{align*}
	t\sum_{j=-1}^{a}\sum_{k=0}^{a}T_j^{-1}\cdots T_0^{-1}\cdot T_k^{-1}\cdots T_0^{-1}&=t\sum_{0\leq j+1\leq k\leq a}T_k^{-1}\cdots T_0^{-1}T_{j+1}^{-1}\cdots T_1^{-1}\\
	&+(1-t)\sum_{0\leq k\leq j\leq a}T_j^{-1}\cdots T_0^{-1}T_k^{-1}\cdots T_1^{-1}\\
	&+\sum_{0\leq k\leq j\leq a}T_{k-1}^{-1}\cdots T_0^{-1}T_j^{-1}\cdots T_1^{-1}\\
	&=\sum_{j=-1}^a\sum_{k=0}^aT_j^{-1}\cdots T_0^{-1}T_k^{-1}\cdots T_1^{-1}.
\end{align*}
The comparison with (\ref{eqn:PTTP}) yields the desired equality.
\end{proof}

\begin{lemma}\label{Lem_modII}
Suppose that $F(X)\in\bb{Q}_{q,t} [X^{\pm1}]_{(m)}$ satisfies $T_{m-1}\bullet F(X)=tF(X)$. For any $0\leq a<m-1$, we have
\begin{align*}
	(1+t)\widehat{\msf{S}}^{(m)}_{a}\widehat{\msf{S}}^{(m)}_{a}\bullet F(X)=t\widehat{\msf{S}}^{(m)}_{a+1}\widehat{\msf{S}}^{(m)}_{a}\bullet F(X)+\widehat{\msf{S}}^{(m)}_{a}\widehat{\msf{S}}^{(m)}_{a-1}\bullet F(X).
\end{align*}
\end{lemma}

\begin{proof}
If $a=0$, then we have $\widehat{\msf{S}}^{(m)}_{0}=\Pi$ and $\widehat{\msf{S}}^{(m)}_{-1}=0$. Hence, we obtain
\begin{align*}
	(1+t)\widehat{\msf{S}}^{(m)}_{0}\widehat{\msf{S}}^{(m)}_{0}\bullet F(X)-t\widehat{\msf{S}}^{(m)}_{1}\widehat{\msf{S}}^{(m)}_{0}\bullet F(X)&=(1+t)\Pi^2\bullet F(X)-t(1+T_1^{-1})\Pi^2\bullet F(X)\\
	&=(1+t)\Pi^2\bullet F(X)-t\Pi^2(1+T_{m-1}^{-1})\bullet F(X)=0
\end{align*}
by the assumption $T_{m-1}^{-1}\bullet F(X) = t^{-1}F(X)$. 

If $0<a<m-1$, then we have
\begin{align*}
	&\Pi^{-1}\left((1+t)\widehat{\msf{S}}^{(m)}_{a}\widehat{\msf{S}}^{(m)}_{a}-t\widehat{\msf{S}}^{(m)}_{a+1}\widehat{\msf{S}}^{(m)}_{a}-\widehat{\msf{S}}^{(m)}_{a}\widehat{\msf{S}}^{(m)}_{a-1}\right)\Pi^{-1}\\
	&=(1+t)(1+T_0^{-1}+\cdots+T_{a-1}^{-1}\cdots T_0^{-1})(1+T_1^{-1}+\cdots+T_{a}^{-1}\cdots T_1^{-1})\\
	&\hspace{1em}-t(1+T_0^{-1}+\cdots+T_{a}^{-1}\cdots T_0^{-1})(1+T_1^{-1}+\cdots+T_{a}^{-1}\cdots T_1^{-1})\\
	&\hspace{1em}-(1+T_0^{-1}+\cdots+T_{a-1}^{-1}\cdots T_0^{-1})(1+T_1^{-1}+\cdots+T_{a-1}^{-1}\cdots T_1^{-1})\\
	&=(1+T_0^{-1}+\cdots+T_{a-1}^{-1}\cdots T_0^{-1})\cdot T_{a}^{-1}\cdots T_1^{-1}-tT_{a}^{-1}\cdots T_0^{-1}\cdot(1+T_1^{-1}+\cdots+T_{a}^{-1}\cdots T_1^{-1}).
\end{align*}
Since we have $\Pi\bullet F(X)=tT_0^{-1}\Pi\bullet F(X)$, we obtain
\begin{align*}
	T_j^{-1}\cdots T_0^{-1}\cdot T_a^{-1}\cdots T_1^{-1}\Pi\bullet F(X)&=tT_j^{-1} \cdots T_1^{-1} \cdot T_a^{-1}\cdots T_2^{-1}T_0^{-1}T_1^{-1}T_0^{-1}\Pi\bullet F(X)\\
	&=tT_j^{-1}\cdots T_1^{-1} \cdot T_a^{-1}\cdots T_0^{-1}\cdot T_1^{-1}\Pi\bullet F(X)\\
	&=tT_j^{-1}\cdots T_2^{-1} \cdot T_a^{-1}\cdots T_1^{-1}T_2^{-1}T_1^{-1}T_0^{-1}\cdot T_1^{-1}\Pi\bullet F(X)\\
	&=tT_j^{-1}\cdots T_2^{-1} \cdot T_a^{-1}\cdots T_0^{-1}\cdot T_2^{-1}T_1^{-1}\Pi(F)\\
	&=\cdots\\
	&=tT_a^{-1}\cdots T_0^{-1}\cdot T_{j+1}^{-1}\cdots T_1^{-1}\Pi(F)
\end{align*}
for each $-1\leq j<a$ by using Definition~\ref{def:AHA}. This implies the desired equality.
\end{proof}

\subsection{Modular law}

\begin{thm}\label{Thm_modular_law}
For any $F(X)\in\bb{Q}_{q,t}[X^{\pm1}]_{(m)}^{\mf{S}_m}$, the map $\chi_{F}:\bb{E}_n\rightarrow\bb{Q}_{q,t} [X^{\pm1}]_{(m)}^{\mf{S}_m}$ given by 
\begin{align*}
	\chi_{F}(\msf{e})=\msf{S}_{\msf{e}}^{(m)}\bullet F(X)
\end{align*}
satisfies the modular law in the sense of Lemma~\ref{Lem_emodular}.
\end{thm}

\begin{proof}
We first assume that a triple $(\msf{e},\msf{e}',\msf{e}'') \in \bb{E}_n$ satisfies the first condition of Lemma~\ref{Lem_emodular}, where we borrow the index $i$ there. In particular, we have $\msf{e}'(i)=\msf{e}(i)+1$ and $\msf{e}''(i)=\msf{e}(i)-1$ and $\msf{e}(j) < \msf{e}(i)$ for each $j < i$. Let $(\msf{a},\msf{a'},\msf{a}'')$ be the corresponding triple in $\bb{A}_n$. We have $\msf{a}'(i)=\msf{a}(i)-1$ and $\msf{a}''(i)=\msf{a}(i)+1$ by $\msf{a}(j)=j-1-\msf{e}(j)$ for $1\leq j \leq m$. A direct computation yields
\begin{align*}
	(1+t)\widehat{\msf{S}}_{m-1-\msf{a}(i)}^{(m)}-t\widehat{\msf{S}}_{m-1-\msf{a}'(i)}^{(m)}-\widehat{\msf{S}}_{m-1-\msf{a}''(i)}^{(m)}=\begin{cases}
		\left(1-tT_{m-\msf{a}(i)}^{-1}\right)T_{m-\msf{a}(i)-1}^{-1}\cdots T_1^{-1}\Pi&\mbox{ if }\msf{a}(i)<m,\\
		-t\Pi &\mbox{ if }\msf{a}(i)=m,\\
		0&\mbox{ if }\msf{a}(i)>m.
	\end{cases}
\end{align*}

If $\msf{a}(i)>m$, then we have nothing to prove. If $\msf{a}(i)=m$, then $\msf{e}(i-1)<\msf{e}(i)$ implies that $\msf{a}(i-1)>m-1$ and hence $\widehat{\msf{S}}_{m-1-\msf{a}(i-1)}^{(m)}=0$. This implies that the modular law trivially holds in this case. 

Therefore, we may assume $\msf{a}(i)< m$. If $\msf{e}(i)+1<j<i$, then we have
\begin{align*}
	\widehat{\msf{S}}^{(m)}_{m-1-\msf{a}(j)}T_{m+\msf{e}(i)-j}^{-1}=T_{m+\msf{e}(i)-j+1}^{-1}\widehat{\msf{S}}^{(m)}_{m-1-\msf{a}(j)}
\end{align*}
by  
\begin{align*}
	m-\msf{a}(j)=m+\msf{e}(j)-j+1<m+\msf{e}(i)-j+1<m.
\end{align*}
Therefore, by moving the factor $\left(1-tT^{-1}_{m-\msf{a}(i)}\right)=\left(1-tT^{-1}_{m+\msf{e}(i)-i+1}\right)$ to the left, it is enough to check 
\begin{align*}
\widehat{\msf{S}}^{(m)}_{m-1-\msf{a}(\msf{e}(i))}\widehat{\msf{S}}^{(m)}_{m-1-\msf{a}(\msf{e}(i)+1)}(1-tT_{m-1}^{-1})=0.
\end{align*}
By the assumption that $\msf{e}(\msf{e}(i)+1)=\msf{e}(\msf{e}(i))$, we have $m-1-\msf{a}(\msf{e}(i))=m-1-\msf{a}(\msf{e}(i)+1)+1$ and hence this equality follows from Lemma~\ref{Lem_modI}. 

We next assume that a triple $(\msf{e},\msf{e}',\msf{e}'')$ satisfies the second condition of Lemma~\ref{Lem_emodular}. If $\msf{a}(i)=0$, then we have $\msf{e}(i)=i-1$ and hence $\msf{e}(i+1)=i$, but this contradicts to the condition $\msf{e}^{-1}(i)=\emptyset$. Hence we should have $\msf{a}(i)>0$.

By Lemma~\ref{Lem_modII} applied to $a=m-1-\msf{a}(i)<m-1$, it is enough to show that $F_{i+2}(X)$ is $s_{m-1}$-invariant, where $F_j(X)$ is defined by (\ref{Eqn_Fj}) in the proof of Lemma~\ref{Lem_Sesym}. 

We inductively define a nondecreasing sequence of integers $k_1,\ldots,k_{n-i}$ by $k_1=0$ and 
\begin{align*}
	k_{j+1}=\begin{cases}
		k_j&\mbox{ if }\msf{a}(i+j+1)<k_j,\\
		k_j+1&\mbox{ if }\msf{a}(i+j+1)\geq k_j.
	\end{cases}
\end{align*}

We note that $\msf{a}(i+j+1)\geq k_{j+1}$ for some $j$ implies $k_{j+1}=k_j+1$. Since $\msf{a}(i+j+1)\geq k_j+1$, we obtain $k_j\leq\msf{a}(i+j+1)-1\leq\msf{a}(i+j)$ and hence also $k_{j}=k_{j-1}+1$. By continuing, we obtain $k_j=j-1$. Therefore, if there exists $j$ such that $\msf{a}(i+j+1)=k_j+1$, then we obtain $\msf{e}(i+j+1)=i+j-\msf{a}(i+j+1)=i$, which contradicts to the assumption that $\msf{e}^{-1}(i)=\emptyset$. Hence we have $\msf{a}(i+j+1)\neq k_j+1$ for any $j$.

We claim that $F_{i+j+1}$ is invariant under $s_{m-1-l_j},s_{m-l_j},\ldots,s_{m-1-k_j}$ for 
\begin{align*}
	l_j\coloneqq\min\{j-1,m-2\}.
\end{align*}
We prove this claim by descending induction on $j$.

If $j=n-i$, then we have $F_{i+j+1}(X)=F(X)$, $m-1-l_{j}\geq1$, and $m-1-k_j\leq m-1$. Hence the assumption implies the claim for $j=n-i$.

Assume by induction that $F_{i+j+2}(X)$ is invariant under $s_{m-1-l_{j+1}},s_{m-l_{j+1}},\ldots,s_{m-1-k_{j+1}}$. Recall that we have
\begin{align*}
	F_{i+j+1}(X)=\widehat{\msf{S}}^{(m)}_{m-1-\msf{a}(i+j+1)}\bullet F_{i+j+2}(X).
\end{align*}

We first consider the case $j\geq m-1$ and hence $m-1-l_j=1$. If we have $\msf{a}(i+j+1)>k_{j+1}$, then we have $k_{j}=j-1$. Since $m-1-k_j=m-j\leq1$, we have nothing to prove if $j\geq m$ and $F_{i+j+1}(X)$ is $s_1$-invariant if $j=m-1$. In the latter case, we have $m-1-\msf{a}(i+j+1)<m-1-j=0$ and hence $F_{i+j+1}(X)=0$. 

If we have $\msf{a}(i+j+1)<k_{j+1}$, then we obtain $\msf{a}(i+j+1)\leq k_{j}$. If $\msf{a}(i+j+1)=k_{j}$, then we have $k_{j+1}=k_j+1$ and $m-1-\msf{a}(i+j+1)=m-1-k_j$. If $\msf{a}(i+j+1)=m-1$, then we have nothing to prove. Hence Lemma~\ref{Lem_Sinv} implies that $F_{i+j+1}(X)$ is invariant under $s_1,\ldots,s_{m-1-k_j}=s_{m-k_{j+1}}$. 

If $\msf{a}(i+j+1)<k_{j}$, then we have $k_{j+1}=k_j$ and $m-1-\msf{a}(i+j+1)>m-1-k_j$. If $\msf{a}(i+j+1)=m-1$, then $m-1-k_j<0$ and we have nothing to prove. Hence Lemma~\ref{Lem_Sinv} implies that $F_{i+j+1}(X)$ is invariant under $s_1,\ldots,s_{m-1-k_j}=s_{m-1-k_{j+1}}$. 

We next consider the case $j<m-1$. In this case, we have $l_j=j-1$ and $l_{j+1}=j$. Hence $F_{i+j+2}(X)$ is invariant under $s_{m-1-j},s_{m-j},\ldots$. Lemma~\ref{Lem_Sinv} and $\msf{a}(i+j+1)\neq k_j+1$ then imply that $F_{i+j+1}(X)$ is invariant under $s_{m-j},\ldots$. This completes the proof of the claim.
  
By applying the claim for $j=1$, we find that $F_{i+2}(X)$ is invariant under $s_{m-1}$ if $F(X)$ symmetric. This finishes the proof of the Theorem. 
\end{proof}

\section{Main theorems}

We now give a proof of our main results stated in the Introduction.

\subsection{Proof of Theorem~\ref{Main_B}}

In this section, we completes the proof of Theorem~\ref{Main_B}. The first assertion is a combination of Lemma~\ref{Lem_iota} and Corollary~\ref{Cor_q_stability}. The second assertion is Proposition~\ref{Prop_qmult_q=1}. The fourth assertion follows from Lemma~\ref{Lem_iota_elem} and the definition of the quantum multiplication.

We now prove the third assertion.

\begin{thm}\label{Thm_main}
For any unit interval graph $\Gamma$, the symmetric function $\msf{q}^{-1}\left(\mb{X}_{\Gamma}(q,t)\right)$ coincides with the chromatic quasisymmetric function $\mb{Y}_{\Gamma}(t)$ written in the $Y$-variables.
\end{thm}

\begin{proof}
By Theorem~\ref{Thm_modular_law}, $\msf{q}^{-1}\left(\mb{X}_{\Gamma_{\msf{e}}}(q,t)\right)$ satisfies the modular law. By Corollary~\ref{Cor_fact}, we obtain 
\begin{align*}
	\msf{q}^{-1}\left(\mb{X}_{\Gamma\cup\Gamma'}(q,t)\right)=\msf{q}^{-1}\left(\mb{X}_{\Gamma}(q,t)\right)\cdot\msf{q}^{-1}\left(\mb{X}_{\Gamma'}(q,t)\right).
\end{align*}
By Lemma~\ref{Lem_iota_elem} and Corollary~\ref{Cor_complete}, we obtain 
\begin{align*}
	\msf{q}^{-1}\left(\mb{X}_{\Gamma_{\msf{e}_n}}(q,t)\right)=[n]_t!\cdot e_n(X).
\end{align*}
Therefore, the result follows from Theorem~\ref{Thm_AN}.
\end{proof}

\subsection{Proof of Theorem~\ref{Main_A}}

In this section, we completes the proof of Theorem~\ref{Main_A}. The first assertion follows from Corollary~\ref{Cor_symm}. The statement about the coefficients is easy to see by the definition. The third assertion follows from Theorem~\ref{Thm_main} by specializing $q=1$.

We now prove the second assertion.

\begin{thm}\label{Thm_stability}
For any $F(X)\in\Lambda_{q,t}^{(m)}$, $\msf{e}\in\bb{E}_n$, and $0< m'<m$, we have
\begin{align*}
	\pi_{m,m'}\left(\msf{S}_{\msf{e}}^{(m)}\bullet F(X)\right)=\msf{S}_{\msf{e}}^{(m')}\bullet \pi_{m,m'}(F(X)).
\end{align*}
\end{thm}

\begin{proof}
By Theorem~\ref{Thm_modular_law}, the both sides of the equality satisfy the modular law with respect to $\msf{e} \in \bb{E}_{n}$. Hence by Remark~\ref{Rem_AN}, it is enough to show the equality when $\msf{e}=\msf{e}_{\mu}$ for any composition $\mu=(\mu_1,\ldots,\mu_l)$ of $n$. This follows by combining Lemma~\ref{Lem_star_commute}, Corollary~\ref{Cor_spmult}, Corollary~\ref{Cor_fact}, and Corollary~\ref{Cor_complete}. 
\end{proof}

\subsection{Proof of Theorem~\ref{Main_C}}

In this final section, we consider the specialization of $(q,t)$-chromatic symmetric functions at $q=\infty$ and prove Theorem~\ref{Main_C}.

We first prove the second assertion of Theorem~\ref{Main_C}.

\begin{lemma}\label{Lem_qtelem_limit}
	For any partition $\lambda=(\lambda_1,\ldots,\lambda_l)$, we have
\begin{align*}
	\lim_{q\rightarrow\infty}e_{\lambda}^{(q,t)}(X)=\frac{[n]_t!}{\prod_{i=1}^{l}[\lambda_i]_t!}\cdot e_{n}(X).
\end{align*}
\end{lemma}

\begin{proof}
By Theorem~\ref{Thm_qtPieri_en}, we obtain
\begin{align*}
	\lim_{q\rightarrow\infty}e_1(X)^{\star r}=\frac{e_r(X)}{[r]_t!}.
\end{align*}
The result follows from this equality.
\end{proof}

We next prove the first assertion of Theorem~\ref{Main_C}.

\begin{prop}\label{Prop_limit}
For any $\msf{e}\in\bb{E}_n$, we have
\begin{align*}
	\lim_{q\rightarrow\infty}\mb{X}_{\Gamma_{\msf{e}}}(q,t)=t^{n(n-1)/2-|\msf{e}|}[n]_t!e_n(X),
\end{align*}
where we set $|\msf{e}|=\sum_{i=1}^{n}\msf{e}(i)$.
\end{prop}

\begin{proof}
The RHS satisfies the modular law since we have $|\msf{e}'| = |\msf{e}| + 1$ and $|\msf{e}''| = |\msf{e}| - 1$ in both cases of Lemma~\ref{Lem_emodular}. Hence, it is enough to check the equality when $\msf{e}=\msf{e}_{\mu}$ for any composition $\mu=(\mu_1,\ldots,\mu_l)$ of $n$. 

We have $|\msf{e}_{\mu}|=\sum_{1\leq i<j\leq l}\mu_i\mu_j$ by Definition~\ref{Def_concat}. In particular, we have
\begin{align*}
	\sum_{i=1}^l \frac{\mu_i(\mu_i-1)}{2} = \frac{n(n-1)}{2}- |\msf{e}_{\mu}|.
\end{align*}
Thus, we conclude
\begin{align*}
	\lim_{q\rightarrow\infty}\mb{X}_{\Gamma_{\msf{e}_{\mu}}}(q,t)=\prod_{i=1}^{l}t^{\mu_i(\mu_i-1)/2}[\mu_i]_t!\cdot \lim_{q\rightarrow\infty}e^{(q,t)}_{\mu}(X)=t^{n(n-1)/2-|\mb{e}_{\mu}|}[n]_t!\cdot e_n(X)
\end{align*}
by Lemma~\ref{Lem_qtelem_limit} as required.
\end{proof}

For a unit interval graph $\Gamma$, we expand $\mb{Y}_{\Gamma}(t)$ as
\begin{align*}
	\mb{Y}_{\Gamma}(t) =\sum_{\lambda\vdash n}c_{\lambda}(\Gamma;t)e_\lambda(Y),
\end{align*}
where $\lambda \vdash n$ means $\lambda$ runs over all partitions of $n$. As a corollary of Proposition~\ref{Prop_limit}, we obtain the following linear relation between $c_{\lambda}(\Gamma;t)$.

\begin{corollary}\label{Cor_dist}
	For any unit interval graph $\Gamma$ corresponding to $\msf{e}\in \bb{E}_n$, we have
\begin{align*}
		\sum_{\lambda\vdash n}t^{|\msf{e}|-|\msf{e}_{\lambda}|}\frac{c_{\lambda}(\Gamma;t)}{\prod_{i}[\lambda_i]_t!}=1.
\end{align*}
\end{corollary}

\begin{proof}
	By Lemma~\ref{Lem_iota_elem} and Theorem~\ref{Thm_main}, we obtain
\begin{align*}
	\mb{X}_{\Gamma}(q,t)=\sum_{\lambda\vdash n}c_{\lambda}(\Gamma;t)t^{\sum_{i}\lambda_i(\lambda_i-1)/2}e_{\lambda}^{(q,t)}(X).
\end{align*}
By taking the limit $q\rightarrow\infty$, and use Lemma~\ref{Lem_qtelem_limit} and Proposition~\ref{Prop_limit}, we obtain the result.
\end{proof}

\bibliography{qt-CSF}
\bibliographystyle{plain}

\end{document}